\documentclass{amsart}
\usepackage[latin1]{inputenc}
\usepackage[T1]{fontenc}
\usepackage{mathrsfs,amssymb,amsthm,amsmath,verbatim}
\usepackage[usenames]{color}

\title{Selective games on binary relations}

\author[R. R. Dias]{Rodrigo R. Dias$^{\ddagger}$}
\thanks{$^{\ddagger}$ Supported by FAPESP (2013/10363-3)}
\address{Instituto de Matem\'atica e Estat\'istica,
Universidade de S\~ao Paulo, Caixa Postal 66281,
S\~ao Paulo, SP, 05315-970, Brazil}
\email{roque@ime.usp.br}

\author[M. Scheepers]{Marion Scheepers}
\address{Department of Mathematics, Boise State University, Boise, ID,
  83725, United States}
\email{mscheepe@boisestate.edu}

\keywords{topological games, selection principles, binary relations,
  product spaces}

\subjclass[2010]{Primary 91A44;
  Secondary 54B10, 54D20, 54D45, 54D65, 54D99}

\begin{document}

\newtheorem{defin}{Definition}[section]

\newtheorem{prop}[defin]{Proposition}

\newtheorem{prob}[defin]{Problem}

\newtheorem{lemma}[defin]{Lemma}

\newtheorem{corol}[defin]{Corollary}

\newtheorem{example}[defin]{Example}

\newtheorem{thm}[defin]{Theorem}

\newtheorem{remark}[defin]{Remark}

\begin{abstract}

We present a unified approach, based on dominating families in binary
relations, for the study of topological properties defined in terms of
selection principles and the games associated to them.

\end{abstract}

\maketitle

\section*{Introduction}

Classical games introduced by Berner and Juh\'asz \cite{berner},
Galvin \cite{galvin}, Gruenhage \cite{gruen76}
and Telg\'arsky \cite{telg1,telg2} associated with diverse topological
properties,  as well as several subsequent games, can be considered in
a single unifying framework based on the notion of a \emph{relation},
defined below in Definition \ref{def.relation}, and of a
\emph{dominating family} for a relation, defined  below in Definition
\ref{def.Dom}. We define these games in Definition
\ref{def.Pgame}. This framework subsumes and clarifies several
isolated theorems about topological games.\footnote{Which
specific theorems in the literature are affected will be presented
later in the paper.}

One of the two main phenomena about these topological games that is
addressed here is: Consider a topological property, $\mathsf{E}$. For
many examples of $\mathsf{E}$ one can find spaces $X$ and $Y$ each of
which has the property $\mathsf{E}$, while the product space $X\times
Y$ does not have the property $\mathsf{E}$. Define a space $X$ to be
\emph{productively $\mathsf{E}$} if, for each space $Y$ that has
property $\mathsf{E}$, also $X\times Y$ has the property
$\mathsf{E}$. For some properties $\mathsf{E}$ it is a significant
mathematical problem to characterize the spaces $X$ that are
productively $\mathsf{E}$. For several isolated examples of
topological games $\mathsf{G}$ it has been found that, if
a certain player of the game ${\sf G}$ on a space $X$
has a winning strategy, then $X$ is productively $\mathsf{E}$.

We study the productivity of properties $\mathsf{E}$ in this abstract
context. The four properties $\mathsf{E}$ we consider
--- which will be described in detail after
Definition \ref{def.Gfin} ---
are as follows:

\begin{itemize}

\item[$\mathsf{E}_1$:]
TWO
has a winning strategy in the game $\mathsf{G}$;

\item[$\mathsf{E}_2$:]
ONE
does not have a winning strategy in the game $\mathsf{G}$;

\item[$\mathsf{E}_3$:]
A selective version of a certain countability hypothesis holds;

\item[$\mathsf{E}_4$:]
A certain countability hypothesis holds.

\end{itemize}

For the properties we consider it will be the case that 
\[
  \mathsf{E}_1\Rightarrow \mathsf{E}_2 \Rightarrow \mathsf{E}_3 \Rightarrow \mathsf{E}_4,
\]
where sometimes, but not always, an implication is reversible.

The nature of our theorems is as follows:

\begin{itemize}

\item[$\cdot$]
A relation in the class $\mathsf{E}_1$ is productively $\mathsf{E}_1$.

\item[$\cdot$]
A relation in the class $\mathsf{E}_1$ is productively $\mathsf{E}_3$.

\item[$\cdot$]
A relation in the class $\mathsf{E}_1$ is productively $\mathsf{E}_4$.

\end{itemize}

It is curious that as of yet the techniques for the other implications
do not seem to produce the implication that a space in the class
$\mathsf{E}_1$ is productively $\mathsf{E}_2$.
Some questions related to this issue will be posed at the end of the
paper.

These results about products bring us to analyzing situations in which
TWO has a winning strategy in a game, and the second of the two main 
phenomena addressed here: In many instances of games $\mathsf{G}$ it
is known that, if player TWO has a winning strategy in the game
$\mathsf{G}$, then player TWO has a winning strategy in a game
$\mathsf{G}^{\prime}$, in which the winning condition for TWO appears
more stringent.


This paper is organized as follows: After establishing notational
conventions in Section 1 we introduce a general framework for the theory regarding
$\mathsf{E}_1$ in Section 2. In Section 3 we translate classical duality results on games
to the new framework. In Sections 4 and 5 we prove product
theorems.
Sections 6 and 7 are dedicated to situations in which the
existence of a winning strategy for player TWO in a certain game turns
out to be equivalent to the same condition in other games that are
seemingly more difficult for TWO.
In Section 8 we study conditions under which a Lindel\"of-like
property is equivalent to ONE not having a winning strategy in the
selective game associated to the property being considered.
Section 9 contains some final remarks about the results presented in
the paper.


\section{Notational conventions}

Throughout our paper $X$ denotes the underlying set of a topological
space and $\tau$ denotes the ambient topology on the space. Whenever a
second topological space $Y$ is involved, its topology will be denoted
by $\rho$.
Unless explicitly stated otherwise, we do not make any assumptions
about separation hypotheses on the topological spaces in our results.

For a space $X$ and a point $x\in X$, we write
$\tau_x=\{U\in\tau:\, x\in U\}$. For a subset $A$ of $X$,
the closure of $A$ in $X$ is denoted by $\overline{A}$. 
The set of all compact subsets of $X$ is denoted by $K(X)$.

A set $A$ is \emph{countable} if $|A|\le\aleph_0$.
For a cardinal number $\lambda$, we write
$\lbrack A\rbrack^{<\lambda}=\{B\subseteq A:\, \vert
  B\vert<\lambda\}$,
$\lbrack A\rbrack^{\lambda}=\{B\subseteq A:\, \vert B\vert=\lambda\}$
and
$\lbrack A\rbrack^{\le\lambda}=\lbrack A\rbrack^{<\lambda}\cup
\lbrack A\rbrack^{\lambda}$. The set of all functions from a set $A$
to a set $B$ is denoted by $\mbox{}^AB$; we also write
$\mbox{}^{<\omega}B$ for $\bigcup_{n\in\omega}\mbox{}^nB$.

Throughout the paper we will make use of several families associated
to a topological space $X$; the reader is referred to Example
\ref{ex.Dom} for both their definition and the notation we adopt to
denote these families.

For definitions of concepts found in the paper that are
neither listed here nor defined right before the result in which they
appear,
the reader is referred to \cite{eng}, \cite{hrbacek} and \cite{kunen}.

We end this section with two definitions we shall make use of
frequently in the paper.

\begin{defin}
\label{def.Lindelof}

Let $\mathcal{A}$ and $\mathcal{B}$ be families of sets.
We say that
\emph{$(\mathcal{A},\mathcal{B})$-Lindel\"of} holds if every element
of $\mathcal{A}$ has a countable subset that is an element of
$\mathcal{B}$. If $\mathcal{A}=\mathcal{B}$, we will say
\emph{$\mathcal{A}$-Lindel\"of} instead of
$(\mathcal{A},\mathcal{B})$-Lindel\"of.

\end{defin}

\begin{defin}
\label{def.S1-Sfin}

Let $\mathcal{A}$ and $\mathcal{B}$ be families of sets.

\begin{itemize}

\item[$\cdot$]
The notation $\mathsf{S}_1(\mathcal{A},\mathcal{B})$ abbreviates the
following statement:
\begin{quote}
For every sequence $(A_n)_{n\in\omega}$ of elements of $\mathcal{A}$,
there is a sequence $(B_n)_{n\in\omega}$ such that $B_n\in A_n$ for
all $n\in\omega$ and $\{B_n:n\in\omega\}\in\mathcal{B}$.
\end{quote}

\item[$\cdot$]
The notation $\mathsf{S}_{\mathrm{fin}}(\mathcal{A},\mathcal{B})$
abbreviates the following statement:
\begin{quote}
For every sequence $(A_n)_{n\in\omega}$ of elements of $\mathcal{A}$,
there is a sequence $(F_n)_{n\in\omega}$ such that $F_n\in
[A_n]^{<\aleph_0}$ for all $n\in\omega$ and
$\bigcup_{n\in\omega}F_n\in\mathcal{B}$.
\end{quote}

\end{itemize}

\end{defin}

Note that $\mathsf{S}_1(\mathcal{A},\mathcal{B})$ implies
$\mathsf{S}_{\mathrm{fin}}(\mathcal{A},\mathcal{B})$, which in turn
implies $(\mathcal{A},\mathcal{B})$-Lindel\"of.

\section{Relations and dominating families}

All of our main results in this paper are phrased in terms of
\emph{dominating families in binary relations}, a general framework
that allows us to express a number of topological concepts through a
unified terminology. This section is dedicated to stating the basic
definitions and exploring how this framework can be used to
capture properties of interest in general topology.

The following definition is based on \cite{vojtas}; see also
\cite[Section 4]{blass}.

\begin{defin}\label{def.relation}

A \emph{relation} is a triple $(A,B,R)$, where
$A\neq\emptyset$ and 
$R\subseteq A\times B$ is such that 
$\forall a\in A\,\exists b\in B\;((a,b)\in R)$. 
\end{defin}

We will henceforth adopt the convention of writing $aRb$ instead of $(a,b)\in R$, and we will read
this as ``$b$ \emph{dominates} $a$ \emph{in} $R$''.

\begin{defin}
\label{def.Dom}
For a relation $P=(A,B,R)$, define
$$\mathrm{Dom}(P)=\{Z\subseteq B:\forall a\in A\;\exists
b\in Z\;(aRb)\}.$$
The elements of $\mathrm{Dom}(P)$ are said to be \emph{dominating in
$P$}.

\end{defin}

The table in Example \ref{ex.Dom} has double purpose: It serves to
illustrate instances of relations whose set of dominating
families have topological meaning, and also to define the concepts and
terminology that are listed therein.

\begin{example}
\label{ex.Dom}

Let $X$ be a topological space with topology $\tau$, and let $x\in
X$. The following table lists several examples of families associated
to $X$ (and $x$) that can be expressed as the set $\mathrm{Dom}(P)$
for some relation $P$.\\ \\
{\renewcommand{\arraystretch}{1.5}
\renewcommand{\tabcolsep}{0.2cm}
\begin{tabular}{|c|c|}

\hline

$P$ & $\mathrm{Dom}(P)$ \\

\hline \hline

$(X,\tau,\in)$ &
$\mathcal{O}_X=\{\mathcal{U}\subseteq\tau:X=\bigcup\mathcal{U}\}$ \\

\hline

$([X]^{<\aleph_0},\tau,\subseteq)$ &
$\Omega_X=\{\mathcal{U}\subseteq\tau:\mathcal{U}\textrm{ is an
}\omega\textrm{-cover of }X\}$ \\

\hline

$(K(X),\tau,\subseteq)$
&
$\mathcal{K}_X=\{\mathcal{U}\subseteq\tau:\mathcal{U}\textrm{ is a
}k\textrm{-cover of }X\}$ \\

\hline

{\renewcommand{\arraystretch}{1.2}
\renewcommand{\tabcolsep}{0cm}
\begin{tabular}{c}
$(K(X),K(X)\setminus\{\emptyset\},R)$,\\
where $HRK\leftrightarrow H\cap K=\emptyset$
\end{tabular}}
&
$\mathfrak{M}_X=\{\mathcal{M}\subseteq
K(X)\setminus\{\emptyset\}:\mathcal{M}$ is a moving-off family$\}$\\

\hline

$(\tau\setminus\{\emptyset\},X,\ni)$ &
$\mathfrak{D}_X=\{D\subseteq X:X=\overline{D}\}$ \\

\hline

{\renewcommand{\arraystretch}{1.2}
\renewcommand{\tabcolsep}{0cm}
\begin{tabular}{c}
$(\tau\setminus\{\emptyset\},\tau,R)$,\\
where $URV\leftrightarrow U\cap V\neq\emptyset$
\end{tabular}}
&
$\mathcal{D}_X=\{\mathcal{U}\subseteq\tau:X=\overline{\bigcup\mathcal{U}}\}$\\

\hline

$(\tau\setminus\{\emptyset\},\tau\setminus\{\emptyset\},\supseteq)$ &
$\Pi_X=\{\mathcal{V}\subseteq\tau\setminus\{\emptyset\}:\mathcal{V}$
is a $\pi$-base for $X\}$ \\

\hline

{\renewcommand{\arraystretch}{1.2}
\renewcommand{\tabcolsep}{0cm}
\begin{tabular}{c}
$(\{(a,U):a\in U\in\tau\},\tau,R)$, \\
where $(a,U)RV\leftrightarrow a\in V\subseteq U$
\end{tabular}}
&
$\mathfrak{B}_X=\{\mathcal{B}\subseteq\tau:\mathcal{B}$ is a base for $X\}$\\

\hline

$(\tau_x,\tau_x,\supseteq)$
&
$\mathfrak{V}_{x}=\{\mathcal{V}\subseteq\tau_x:\mathcal{V}$ is a local
base for $X$ at $x\}$\\

\hline

{\renewcommand{\arraystretch}{1.2}
\renewcommand{\tabcolsep}{0cm}
\begin{tabular}{c}
$(X\setminus\{x\},\tau_x,\notin)$\\
(assuming $(X,\tau)$ is $T_1$)
\end{tabular}}
&
$\Psi_{x}=\{\mathcal{V}\subseteq\tau_x:\mathcal{V}\neq\emptyset$ and
$\bigcap\mathcal{V}=\{x\}\}$\\

\hline

$(\tau_x,X,\ni)$
&
$\Omega_{x}=\{A\subseteq X:x\in\overline{A}\}$ \\

\hline

$(\tau_x,[X]^{<\aleph_0}\setminus\{\emptyset\},\supseteq)$
&
$\pi\mathcal{N}_{x}=\{\mathcal{S}\subseteq[X]^{<\aleph_0}\setminus\{\emptyset\}:\mathcal{S}$
is a $\pi$-network at $x\}$ \\

\hline

$(\tau_x,[X]^{\le\aleph_0}\setminus\{\emptyset\},\supseteq)$
&
$\pi\mathcal{N}^{\aleph_0}_{x}=\{\mathcal{S}\subseteq[X]^{\le\aleph_0}\setminus\{\emptyset\}:\mathcal{S}$
is a $\pi$-network at $x\}$ \\

\hline

\end{tabular}}\\

\end{example}

Note that several classical topological properties can be phrased in
terms of families of the form $\mathrm{Dom}(P)$: For example, a
topological space $X$ is Lindel\"of if and only if
$\mathcal{O}_X$-Lindel\"of holds, and $X$ is countably tight at a
point $x\in X$ if and only if $\Omega_x$-Lindel\"of holds.

\section{Games and duality}

We now proceed to defining the basic games we shall consider in this
paper.

\begin{defin}
\label{equivalent.games}

We say that two games $G$ and $G^{\prime}$ are \emph{equivalent} if
both of the following hold:

\begin{itemize}

\item[$\cdot$]
ONE has a winning strategy in the game $G$ if, and only if,
ONE has a winning strategy in $G^{\prime}$; and

\item[$\cdot$]
TWO has a winning strategy in the game $G$ if, and only if,
TWO has a winning strategy in $G^{\prime}$.

\end{itemize}

\end{defin}

\begin{defin}
\label{def.dual}

We say that two games $G$ and $G^{\prime}$ are \emph{dual games} if
both of the following hold:

\begin{itemize}

\item[$\cdot$]
ONE has a winning strategy in the game $G$ if, and only if,
TWO has a winning strategy in $G^{\prime}$; and

\item[$\cdot$]
TWO has a winning strategy in the game $G$ if, and only
if, ONE has a winning strategy in $G^{\prime}$.

\end{itemize}

\end{defin}

The first game we consider is a natural game associated with a pair of
relations:

\begin{defin}
\label{def.Pgame}

Let $P=(A,B,R)$ and $Q=(C,D,T)$ be relations. The game
$\mathsf{G}(P,Q)$ is defined as follows: In each
inning $n\in\omega$, player ONE chooses $a_n\in A$, and
then player TWO chooses $b_n\in B$ with $a_n R b_n$. ONE
wins a play
\[
  a_0,\, b_0,\, \cdots, a_n,\, b_n\, \cdots
\]
 if
$\{b_n:n\in\omega\}\in\mathrm{Dom}(Q)$. Otherwise, TWO wins.

\end{defin}

As the following examples show, there are several topological games
studied in the literature that can be regarded as instances of the
game $\mathsf{G}(P,Q)$ introduced above.

\begin{example}

Let $P$ be the relation $(X,\tau,\in)$. Then
$\mathrm{Dom}(P)$ is $\mathcal{O}_X$, the collection of open covers of
$X$. In this instance the game $\mathsf{G}(P,P)$ corresponds to the
point-open game $\mathsf{G}(X)$ of Galvin, introduced in
\cite{galvin}.

\end{example}

\begin{example}
\label{game.berner}

Let $P$ be the relation
$(\tau\setminus\{\emptyset\},X,\ni)$. Then $\mathrm{Dom}(P)$ is
$\mathfrak{D}_X$, the collection of dense subsets of $X$. In this
instance the game $\mathsf{G}(P,P)$ corresponds to the point-picking
game $\mathsf{G}^D_{\omega}(X)$ of Berner and Juh\'asz, introduced in
\cite{berner}.

\end{example}

\begin{example}
\label{game.gruen}

Fix a point $x\in X$. Let
$P$ be the relation $(\tau_x,X ,\ni)$. Then $\mathrm{Dom}(P)$ is
$\Omega_x$, the collection of subsets of $X$ that have $x$ as a
cluster point. In this instance the game $\mathsf{G}(P,P)$ corresponds
to the game $\mathsf{G}^c_{O,\, P}(X,x)$ of Gruenhage, introduced in
\cite{gruen76}.

\end{example}

\begin{example}
\label{game.tkachuk}

Let $P$ be the relation $(X,\tau,\in)$ and $Q$ be the relation
$(\tau\setminus\{\emptyset\},\tau,T)$, where $UTV\leftrightarrow U\cap
V\neq\emptyset$. Then $\mathrm{Dom}(Q)$ is
$\mathcal{D}_X$, the collection of open families with union dense
in $X$.
In this instance the game $\mathsf{G}(P,Q)$ corresponds to the game
$\theta(X)$ of Tkachuk, introduced in \cite{tkachuk}.

\end{example}

Now, towards the duality theorem mentioned in the introduction, define
the following game:

\begin{defin}
\label{def.G1}

Let $\mathcal{A}$ and $\mathcal{B}$ be families of
sets with $\mathcal{A}\neq\emptyset$. The game
${\sf G}_1(\mathcal{A},\mathcal{B})$ is played as follows. In each
inning $n\in \omega$, ONE first chooses a set $A_n\in\mathcal{A}$, and
then TWO responds with a $B_n\in A_n$. A play
\[
  A_0,\, B_0,\, \cdots,\, A_n,\, B_n,\,\cdots
\]
is won by TWO if $\{B_n:n\in\omega\}$ is an element of
$\mathcal{B}$. Otherwise, ONE wins.

\end{defin}

We shall also refer to the following variation of the previous game
later on.

\begin{defin}
\label{def.Gfin}

Let $\mathcal{A}$ and $\mathcal{B}$ be families of
sets with $\mathcal{A}\neq\emptyset$. The game
${\sf G}_{\mathrm{fin}}(\mathcal{A},\mathcal{B})$ is played as
follows. In each inning $n\in \omega$, ONE first chooses a set
$A_n\in\mathcal{A}$, and then TWO responds with a finite subset $F_n$
of $A_n$. A play
\[
  A_0,\, F_0,\, \cdots,\, A_n,\, F_n,\,\cdots
\]
is won by TWO if $\bigcup_{n\in\omega}F_n$ is an element of
$\mathcal{B}$. Otherwise, ONE wins.

\end{defin}

It is immediate that, if TWO has a winning strategy in
$\mathsf{G}_1(\mathcal{A},\mathcal{B})$, then TWO has a winning
strategy in
$\mathsf{G}_{\mathrm{fin}}(\mathcal{A},\mathcal{B})$. Similarly, if
ONE has a winning strategy in
$\mathsf{G}_{\mathrm{fin}}(\mathcal{A},\mathcal{B})$, then ONE has a
winning strategy in $\mathsf{G}_1(\mathcal{A},\mathcal{B})$.

Furthermore, we have the following chain of implications (now defining
in more detail the four properties $\mathsf{E}_1$--$\mathsf{E}_4$
considered in the Introduction):
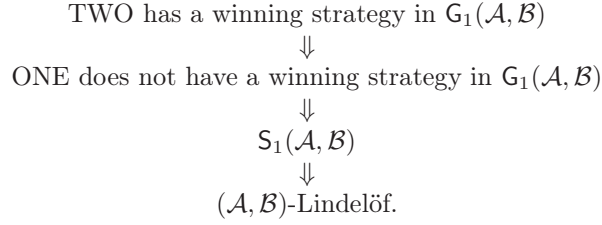
\begin{center}
TWO has a winning strategy in
$\mathsf{G}_1(\mathcal{A},\mathcal{B})$\\
$\Downarrow$\\
ONE does not have a winning strategy in
$\mathsf{G}_1(\mathcal{A},\mathcal{B})$\\
$\Downarrow$\\
$\mathsf{S}_1(\mathcal{A},\mathcal{B})$\\
$\Downarrow$\\
$(\mathcal{A},\mathcal{B})$-Lindel\"of.\\
\mbox{}
\end{center}
Note that the same implications hold when the subscript ``$1$'' is
replaced with ``$\mathrm{fin}$''.

The following duality result is a rephrasing of Theorem 1 of
\cite{galvin}.

\begin{thm}[Galvin \cite{galvin}]
\label{dual}

For all relations $P$ and $Q$, the games $\mathsf{G}(P,Q)$ and
$\mathsf{G}_1(\mathrm{Dom}(P),\mathrm{Dom}(Q))$ are dual.

\end{thm}

\begin{proof}

We will prove that, if TWO has a winning strategy in
$\mathsf{G}_1(\mathrm{Dom}(P),\mathrm{Dom}(Q))$, then ONE has a
winning strategy in $\mathsf{G}(P,Q)$. The reader shall find no
difficulty in verifying the remaining three implications.

The core of the proof is the following observation: If
$\sigma:(\mbox{}^{<\omega}\mathrm{Dom}(P))\setminus\{\emptyset\}\rightarrow
B$ is a strategy for TWO in
$\mathsf{G}_1(\mathrm{Dom}(P),\mathrm{Dom}(Q))$, then
\begin{equation}
\forall s\in\mbox{}^{<\omega}\mathrm{Dom}(P)\;\exists\, a_s\in A\;(\{b\in
B:a_sRb\}\subseteq\{\sigma(s^\frown(Z)):Z\in\mathrm{Dom}(P)\}). \tag{$\dagger$}
\end{equation}

Suppose, in order to get a contradiction, that there is
$s\in\mbox{}^{<\omega}\mathrm{Dom}(P)$ witnessing the failure of
$(\dagger)$. Then, for each $a\in A$, there is $b_a\in B$ with $aRb_a$
such that there is no $Z\in\mathrm{Dom}(P)$ satisfying
$b_a=\sigma(s^\frown(Z))$. But then we get a contradiction from the
fact that $\widetilde{Z}=\{b_a:a\in A\}\in\mathrm{Dom}(P)$, since
this implies that $\sigma(s^\frown(\widetilde{Z}))=b_a$ for some $a\in
A$.

Having proved $(\dagger)$, let now
$\sigma:(\mbox{}^{<\omega}\mathrm{Dom}(P))\setminus\{\emptyset\}\rightarrow
B$ be a winning strategy for TWO in
$\mathsf{G}_1(\mathrm{Dom}(P),\mathrm{Dom}(Q))$. A winning strategy
for ONE in $\mathsf{G}(P,Q)$ can then be defined as follows.

ONE's initial move is $a_\emptyset\in A$. If $b_0\in B$ is TWO's
response, it follows from $(\dagger)$ that there is
$Z_0\in\mathrm{Dom}(P)$ such that $b_0=\sigma((Z_0))$; let then
$a_{(Z_0)}\in A$ be ONE's next move in the play. TWO will respond with
some $b_1\in B$ satisfying $a_{(Z_0)}Rb_1$; by $(\dagger)$, there is
$Z_1\in\mathrm{Dom}(P)$ satisfying $b_1=\sigma((Z_0,Z_1))$; ONE's next
move will then be $a_{(Z_0,Z_1)}\in A$; and so forth.

By proceeding in this fashion, we obtain a sequence
$(Z_n)_{n\in\omega}$ of elements of $\mathrm{Dom}(P)$ and a play of
$\mathsf{G}(P,Q)$ in which TWO's move in the inning $n\in\omega$ is
$b_n=\sigma((Z_k)_{k\le n})$. Since $\sigma$ is a winning strategy for
TWO in $\mathsf{G}_1(\mathrm{Dom}(P),\mathrm{Dom}(Q))$, it follows
that
$\{b_n:n\in\omega\}=\{\sigma((Z_k)_{k\le n}):n\in\omega\}\in\mathrm{Dom}(Q)$;
therefore, the strategy above described is a winning strategy for ONE
in $\mathsf{G}(P,Q)$.
\end{proof}

Theorem \ref{dual} expresses in terms of relations and dominating
families the underlying argument in duality results such as the
following corollaries.

The first one, in which the argument above was first presented, is
Theorem 1 of \cite{galvin}.
Recall that the \emph{Rothberger game on $X$}
\cite{galvin} is the game $\mathsf{G}_1(\mathcal{O}_X,\mathcal{O}_X)$.
A topological space $X$ is a \emph{Rothberger space}
\cite{roth} if $\mathsf{S}_1(\mathcal{O}_X,\mathcal{O}_X)$ holds.

\begin{corol}[Galvin \cite{galvin}]\label{cor.rothberger}
For every topological space $X$, the
point-open game on $X$ and the Rothberger game on $X$
are dual games.
\end{corol}

\begin{proof}

Apply Theorem
\ref{dual} with $P=Q=(X,\tau,\in)$.
\end{proof}

The next corollary is the combination of Theorems 7 and 8 of
\cite{sch6}.

\begin{corol}[Scheepers \cite{sch6}]\label{cor.seldense}
For every topological space $X$, the games
$\mathsf{G}^D_{\omega}(X)$ and
$\mathsf{G}_1(\mathfrak{D}_X,\mathfrak{D}_X)$ are dual games.
\end{corol}

\begin{proof}

Apply Theorem
\ref{dual} with $P=Q=(\tau\setminus\{\emptyset\},X,\ni)$.
\end{proof}

The next corollary is Theorem 3.3(2) of \cite{tkachuk}.

\begin{corol}[Tkachuk \cite{tkachuk}]
\label{dual.tk}

For a topological space $X$, the games $\theta(X)$ and
$\mathsf{G}_1(\mathcal{O}_X,\mathcal{D}_X)$ are dual games.

\end{corol}

\begin{proof}

Apply Theorem \ref{dual} with $P=(X,\tau,\in)$ and
$Q=(\tau\setminus\{\emptyset\},\tau,T)$,
where $UTV\leftrightarrow U\cap V\neq\emptyset$.
\end{proof}

\begin{corol}[folklore]\label{cor.strcountabletight}
For a topological space $X$ and a point $x\in X$, the games
$\mathsf{G}^c_{O,\, P}(X,x)$ and $\mathsf{G}_1(\Omega_x,\Omega_x)$
on $X$ are dual games.
\end{corol}

\begin{proof}

Apply Theorem
\ref{dual} with $P=Q=(\tau_x,X,\ni)$.
\end{proof}

\begin{remark}
\label{B=D}

It is worth pointing out that, for two relations $P=(A,B,R)$ and
$Q=(C,D,T)$, the games $\mathsf{G}(P,Q)$ and
$\mathsf{G}_1(\mathrm{Dom}(P),\mathrm{Dom}(Q))$ are of interest only
if $B\in\mathrm{Dom}(Q)$, for otherwise TWO (resp. ONE) wins every
play of $\mathsf{G}(P,Q)$
(resp. $\mathsf{G}_1(\mathrm{Dom}(P),\mathrm{Dom}(Q))$) trivially.
Although we do not include this condition in the definition of these
games (since it is not needed for most of our results), the reader
should note that, in all of the topological situations we consider in
this paper, the sets $B$ and $D$ are the same.

\end{remark}

\section{Products and singleton selections}

In this section we explore the behavior of some selective properties
under products, focusing on the game $\mathsf{G}(P,Q)$ --- and,
equivalently, on the game
$\mathsf{G}_1(\mathrm{Dom}(P),\mathrm{Dom}(Q))$.

The first step towards this goal is to define the product of two
relations. Since we are interested in applying the general results to
topological properties, we must consider a definition that allows us
to do the following in as many situations as possible: If $P$ and $Q$
are relations whose sets of dominating families correspond to a
certain topological concept on the topological spaces $X$ and $Y$
respectively, then the set of dominating families in their product
$P\otimes Q$ must allow us to describe the same topological concept in
the product space $X\times Y$. A natural way of defining such product
is the following.

\begin{defin}
\label{def.prod}

The product relation of two relations $P=(A,B,R)$ and $P'=(A',B',R')$
is defined as $P\otimes P'=(A\times A',B\times B',R\otimes R')$, where
$$R\otimes R'=\{((a,a'),(b,b'))\in(A\times
A')\times(B\times B'):a R b\textrm{ and }a' R' b'\}.$$

\end{defin}

\begin{example}
\label{ex.prod1}

Let $P=(\tau\setminus\{\emptyset\},X,\ni)$ and
$Q=(\rho\setminus\{\emptyset\},Y,\ni)$. Then $P\otimes
Q=((\tau\setminus\{\emptyset\})\times(\rho\setminus\{\emptyset\}),X\times
Y,\ni\otimes\ni)$, where $(U,V)\ni\otimes\ni(x,y)\leftrightarrow(U\ni
x\;\&\;V\ni y)$. In this case, we have
$\mathrm{Dom}(P)=\mathfrak{D}_X$, $\mathrm{Dom}(Q)=\mathfrak{D}_Y$ and
$\mathrm{Dom}(P\otimes Q)=\mathfrak{D}_{X\times Y}$.

\end{example}

\begin{example}
\label{ex.prod2}

Let $P=(X,\tau,\in)$ and $Q=(Y,\rho,\in)$. Then $P\otimes
Q=(X\times Y,\tau\times\rho,\in\otimes\in)$, where
$(x,y)\in\otimes\in(U,V)\leftrightarrow(x\in U\;\&\;y\ni V)$. In this
case, we have $\mathrm{Dom}(P)=\mathcal{O}_X$ and
$\mathrm{Dom}(Q)=\mathcal{O}_Y$. The set $\mathrm{Dom}(P\otimes Q)$
does not correspond exactly to $\mathcal{O}_{X\times Y}$, but rather
to the set of covers of the product $X\times Y$ constituted by
\emph{basic} open sets --- which, however, is enough to express
properties such as ``$X\times Y$ is Lindel\"of'' and ``$X\times Y$ is
Rothberger'' as $\mathrm{Dom}(P\otimes Q)$-Lindel\"of and
$\mathsf{S}_1(\mathrm{Dom}(P\otimes Q),\mathrm{Dom}(P\otimes Q))$
respectively.

\end{example}

We can now prove the main result of this section --- which was
described in general terms in the Introduction.

\begin{prop}
\label{ONE}

Let $P$, $P'$, $Q$ and $Q'$ be relations.
Suppose that ONE has a winning strategy in the game
$\mathsf{G}(P',Q')$.

\begin{itemize}

\item[$(a)$]
If $(\mathrm{Dom}(P),\mathrm{Dom}(Q))$-Lindel\"of holds, then
$(\mathrm{Dom}(P\otimes P'),\mathrm{Dom}(Q\otimes Q'))$-Lindel\"of
also holds.

\item[$(b)$]
If $\mathsf{S}_1(\mathrm{Dom}(P),\mathrm{Dom}(Q))$ holds, then
$\mathsf{S}_1(\mathrm{Dom}(P\otimes P'),\mathrm{Dom}(Q\otimes Q'))$
also holds.

\item[$(c)$]
If ONE has a winning strategy in the game
$\mathsf{G}(P,Q)$, then ONE has
a winning strategy in the game
$\mathsf{G}(P\otimes P',Q\otimes Q')$.

\end{itemize}

\end{prop}

\begin{proof}

Write $P=(A,B,R)$, $P'=(A',B',R')$, $Q=(C,D,T)$
and $Q'=(C',D',T')$, and let
$\sigma:\mbox{}^{<\omega}B'\rightarrow A'$ be a winning strategy for
ONE in the game $\mathsf{G}(P',Q')$.\\

$(a)$
Let $\{(b_i,b'_i):i\in I\}\in\mathrm{Dom}(P\otimes P')$ be fixed.
We will construct indexed families $\langle
a'_s:s\in\mbox{}^{<\omega}\omega\rangle$ and $\langle
i^s_n:s\in\mbox{}^{<\omega}\omega,n\in\omega\rangle$ satisfying:
\begin{itemize}
\item[$\cdot$]
$a'_s\in A'$ for all $s\in\mbox{}^{<\omega}\omega$;
\item[$\cdot$]
$i^s_n\in I$ for all $s\in\mbox{}^{<\omega}\omega$ and $n\in\omega$;
\item[$\cdot$]
$\{b_{i^s_n}:n\in\omega\}\in\mathrm{Dom}(Q)$ for all
  $s\in\mbox{}^{<\omega}\omega$; and
\item[$\cdot$]
  for each $f\in\mbox{}^\omega\omega$,
$$
\left(
a'_\emptyset,b'_{i^\emptyset_{f(0)}},
a'_{(f(0))},b'_{i^{(f(0))}_{f(1)}},
a'_{(f(0),f(1))},b'_{i^{(f(0),f(1))}_{f(2)}},
\dots,
a'_{f\upharpoonright k},b'_{i^{f\upharpoonright k}_{f(k)}},
\dots
\right)
$$
is a play of $\mathsf{G}(P',Q')$ in which ONE follows the
strategy $\sigma$.
\end{itemize}

We proceed by recursion. Suppose that $k\in\omega$ is such that
$\langle a'_t:t\in\mbox{}^{<k}\omega\rangle$ and $\langle
i^t_n:t\in\mbox{}^{<k}\omega,n\in\omega\rangle$
have already been constructed, and let $s\in\mbox{}^k\omega$ be
fixed. Define
$a'_s=\sigma\left(\left(b'_{i^{s\upharpoonright j}_{s(j)}}\right)_{j<k}\right)$.
Note that $\{b_i:i\in I\textrm{ and }a'_sR'b'_i\}\in\mathrm{Dom}(P)$;
hence, by our hypothesis, there is
$\{i^s_n:n\in\omega\}\subseteq\{i\in I:a'_sR'b'_i\}$ with
$\{b_{i^s_n}:n\in\omega\}\in\mathrm{Dom}(Q)$. This completes the
recursion.

We now claim that
$\left\{\left(b_{i^s_n},b'_{i^s_n}\right):s\in\mbox{}^{<\omega}\omega,n\in\omega\right\}\in\mathrm{Dom}(Q\otimes
Q')$.
Indeed, let $(c,c')\in C\times C'$ be arbitrary. Pick $n_k\in\omega$
recursively for $k\in\omega$ so that
$cTb_{i^{(n_j)_{j<k}}_{n_k}}$. Now
$$
\left(
a'_\emptyset,b'_{i^\emptyset_{n_0}},
a'_{(n_0)},b'_{i^{(n_0)}_{n_1}},
a'_{(n_0,n_1)},b'_{i^{(n_0,n_1)}_{n_2}},
\dots,
a'_{(n_j)_{j<k}},b'_{i^{(n_j)_{j<k}}_{n_k}},
\dots
\right)
$$
is a play $\mathsf{G}(P',Q')$ in which ONE makes use of the
winning strategy $\sigma$, so there is $k\in\omega$ such that
$c'T'b'_{i^{(n_j)_{j<k}}_{n_k}}$. Hence $(c,c')T\otimes
T'\left(b_{i^{(n_j)_{j<k}}_{n_k}},b'_{i^{(n_j)_{j<k}}_{n_k}}\right)$,
as required.\\

$(b)$
Let $(Z_n)_{n\in\omega}$ be a sequence of elements of
$\mathrm{Dom}(P\otimes P')$, and write
$Z_n=\{(b^n_i,b'^n_i):i\in I_n\}$ for each $n\in\omega$.
Now write $\omega=\{m^s_k:s\in\mbox{}^{<\omega}\omega,k\in\omega\}$
with $m^s_k=m^t_l$ only if $(s,k)=(t,l)$.
We will assign to each $s\in\mbox{}^{<\omega}\omega$
and each $k\in\omega$ an $i^s_k\in I_{m^s_k}$ according to the
following procedure.

Suppose that $h\in\omega$ is such that $i^t_k\in I_{m^t_k}$ has
already been chosen for all $t\in\mbox{}^{<h}\omega$ and all
$k\in\omega$. Now let $s\in\mbox{}^h\omega$ be fixed. Define
$a'_s=\sigma\left(\left(b'^{m^{s\upharpoonright j}_{s(j)}}_{i^{s\upharpoonright j}_{s(j)}}\right)_{j<h}\right)$.
Since $\left\{b_i^{m^s_k}:i\in
I_{m^s_k}\textrm{ and }a'_sR'b'^{m^s_k}_i\right\}\in\mathrm{Dom}(P)$ for each
$k\in\omega$, we may apply
$\mathsf{S}_1(\mathrm{Dom}(P),\mathrm{Dom}(Q))$ to obtain a sequence
$(i^s_k)_{k\in\omega}$ satisfying:
\begin{itemize}
\item[$\cdot$]
$i^s_k\in I_{m^s_k}$ for each $k\in\omega$;
\item[$\cdot$]
$a'_sR'b'^{m^s_k}_{i^s_k}$ for each $k\in\omega$; and
\item[$\cdot$]
$\left\{b_{i^s_k}^{m^s_k}:k\in\omega\right\}\in\mathrm{Dom}(Q)$.
\end{itemize}
By recursion, this concludes the definition of $i^s_k$ for
$s\in\mbox{}^{<\omega}\omega$ and $k\in\omega$. (In each step, what we
have done is: if $s=(k_0,k_1,\dots,k_{h-1})\in\mbox{}^h\omega$,
then $a'_s$ is ONE's move in a play of $\mathsf{G}(P',Q')$
whose history so far is
$$\left(a'_\emptyset,b'^{m^{\emptyset}_{k_0}}_{i^{\emptyset}_{k_0}},
a'_{(k_0)},b'^{m^{(k_0)}_{k_1}}_{i^{(k_0)}_{k_1}},
a'_{(k_0,k_1)},b'^{m^{(k_0,k_1)}_{k_2}}_{i^{(k_0,k_1)}_{k_2}},
\dots,
a'_{(k_0,k_1,\dots,k_{h-2})},b'^{m^{(k_0,k_1,\dots,k_{h-2})}_{k_{h-1}}}_{i^{(k_0,k_1,\dots,k_{h-2})}_{k_{h-1}}}
\right);$$
then, in view of the fact that
$Z_n(a')=\left\{b^n_i:i\in I_n\textrm{ and }a'R'b'^n_i\right\}$ is an
element of $\mathrm{Dom}(P)$ for all $a'\in A'$ and $n\in\omega$, we
have made use of $\mathsf{S}_1(\mathrm{Dom}(P),\mathrm{Dom}(P'))$
to select from each $I_{m^s_k}$ with $k\in\omega$ an element
$b_{i^s_k}^{m^s_k}$ of $Z_{m^s_k}(a'_s)$ in such a way that
$\left\{b_{i^s_k}^{m^s_k}:k\in\omega\right\}\in\mathrm{Dom}(Q)$.)

We now claim that
$\left\{
\left(b_{i^s_k}^{m^s_k},b'^{m^s_k}_{i^s_k}\right):
s\in\mbox{}^{<\omega}\omega,k\in\omega
\right\}\in\mathrm{Dom}(Q\otimes Q')$.
Indeed, let $(c,c')\in C\times C'$ be arbitrary.
Since $\left\{b_{i^s_k}^{m^s_k}:k\in\omega\right\}\in\mathrm{Dom}(Q)$
for every $s\in\mbox{}^{<\omega}\omega$, we may recursively pick, for
each $r\in\omega$, a $k_r\in\omega$ such that
$cTb^{m^{(k_j)_{j<r}}_{k_r}}_{i^{(k_j)_{j<r}}_{k_r}}$.
Then
$$
\left(
a'_\emptyset,b'^{m^{\emptyset}_{k_0}}_{i^{\emptyset}_{k_0}},
a'_{(k_0)},b'^{m^{(k_0)}_{k_1}}_{i^{(k_0)}_{k_1}},
a'_{(k_0,k_1)},b'^{m^{(k_0,k_1)}_{k_2}}_{i^{(k_0,k_1)}_{k_2}},
\dots,
a'_{(k_j)_{j<r}},b'^{m^{(k_j)_{j<r}}_{k_r}}_{i^{(k_j)_{j<r}}_{k_r}},
\dots
\right)
$$
is a play of $\mathsf{G}(P',Q')$ in which ONE follows the
winning strategy $\sigma$, whence there is $r\in\omega$ with
$c'T'b'^{m^{(k_j)_{j<r}}_{k_r}}_{i^{(k_j)_{j<r}}_{k_r}}$. Therefore,
$(c,c')T\otimes T'
\left(b^{m^{(k_j)_{j<r}}_{k_r}}_{i^{(k_j)_{j<r}}_{k_r}},
b'^{m^{(k_j)_{j<r}}_{k_r}}_{i^{(k_j)_{j<r}}_{k_r}}\right)$.\\

$(c)$
Let
$\varphi:\mbox{}^{<\omega}B\rightarrow A$ be a winning strategy for
ONE in $\mathsf{G}(P,Q)$. Consider a partition
$\omega=\dot\bigcup\{L_s:s\in\mbox{}^{<\omega}\omega\}$, where each
$L_s$ is infinite and $\min(L_s)>\max(\mathrm{im}(s))$ for all
$s\in\mbox{}^{<\omega}\omega\setminus\{\emptyset\}$.
For each $s\in\mbox{}^{<\omega}\omega$, write
$L_s=\{m^s_k:k\in\omega\}$ with $m^s_k<m^s_{k+1}$ for all
$k\in\omega$. A strategy for ONE in the game
$\mathsf{G}(P\otimes P',Q\otimes Q')$ will then be defined as follows.

Suppose that, at the inning $n\in\omega$ of a play of
$\mathsf{G}(P\otimes P',Q\otimes Q')$, the sequence of moves so far is
$$((a_0,a'_0),(b_0,b'_0),(a_1,a'_1),(b_1,b'_1),\dots,(a_{n-1},a'_{n-1}),(b_{n-1},b'_{n-1})).$$
Let $s\in\mbox{}^{<\omega}\omega$ and $k\in\omega$ be such that
$n=m^s_k\in L_s$. Then ONE's move in the current inning is
$(a_n,a'_n)\in A\times A'$, where
$a_n=\varphi((b_{m^s_i})_{i<k})$
and
$a'_n=\sigma((b'_{s(j)})_{j<\mathrm{dom}(s)})$. That is: ONE follows
the strategy $\varphi$ in a play of $\mathsf{G}(P,Q)$ whose previous
moves by TWO are not all of $b_0,b_1,\dots,b_{n-1}$, but only those
with indices in the set $L_s\cap n$; in a similar fashion, ONE
makes use of the strategy $\sigma$ in a play of $\mathsf{G}(P',Q')$
whose previous moves by TWO are not all of $b'_0,b'_1,\dots,b'_{n-1}$,
but only those with indices listed by the sequence $s$.

We claim that this is a winning strategy for ONE in
$\mathsf{G}(P\otimes P',Q\otimes Q')$.
Indeed, suppose that
$$((a_0,a'_0),(b_0,b'_0),(a_1,a'_1),(b_1,b'_1),\dots,(a_{n},a'_{n}),(b_{n},b'_{n}),\dots)$$
is a play of $\mathsf{G}(P\otimes P',Q\otimes Q')$ in which ONE
adopts the strategy above described,
and let $(c,c')\in C\times C'$ be arbitrary. We can construct a
strictly increasing infinite sequence
$(n_j)_{j\in\omega}\in\mbox{}^\omega\omega$ by picking recursively
$n_j\in L_{(n_i)_{i<j}}$ with $cTb_{n_j}$ for each $j\in\omega$ ---
such an $n_j$ must exist, for otherwise
$$(a_{m^{(n_i)_{i<j}}_0},b_{m^{(n_i)_{i<j}}_0},a_{m^{(n_i)_{i<j}}_1},b_{m^{(n_i)_{i<j}}_1},\dots,a_{m^{(n_i)_{i<j}}_k},b_{m^{(n_i)_{i<j}}_k},\dots)$$
would be a play of $\mathsf{G}(P,Q)$ in which ONE plays according to
the winning strategy $\varphi$ and loses.
Now
$$
(a'_{n_0},b'_{n_0},a'_{n_1},b'_{n_1},\dots,a'_{n_j},b'_{n_j},\dots)
$$
is a play of $\mathsf{G}(P',Q')$ in which ONE employs the winning
strategy $\sigma$, whence $c'T'b'_{n_k}$ must hold for some
$k\in\omega$.
Thus $(c,c')T\otimes T'(b_{n_k},b'_{n_k})$, as required.
\end{proof}

In view of Theorem \ref{dual}, Proposition \ref{ONE} can be restated
as follows.

\begin{prop}
\label{G1}

Let $P$, $P'$, $Q$ and $Q'$ be relations.
Suppose that TWO has a winning strategy in the game
$\mathsf{G}_1(\mathrm{Dom}(P'),\mathrm{Dom}(Q'))$.

\begin{itemize}

\item[$(a)$]
If $(\mathrm{Dom}(P),\mathrm{Dom}(Q))$-Lindel\"of holds, then
$(\mathrm{Dom}(P\otimes P'),\mathrm{Dom}(Q\otimes Q'))$-Lindel\"of
also holds.

\item[$(b)$]
If $\mathsf{S}_1(\mathrm{Dom}(P),\mathrm{Dom}(Q))$ holds, then
$\mathsf{S}_1(\mathrm{Dom}(P\otimes P'),\mathrm{Dom}(Q\otimes Q'))$
also holds.

\item[$(c)$]
If TWO has a winning strategy in the game
$\mathsf{G}_1(\mathrm{Dom}(P),\mathrm{Dom}(Q))$, then TWO has
a winning strategy in the game
$\mathsf{G}_1(\mathrm{Dom}(P\otimes P'),\mathrm{Dom}(Q\otimes Q'))$.

\end{itemize}

\end{prop}

We now present some instances of the previous propositions.
The following list of results is not meant to 
exhaust the consequences that can be obtained from Propositions
\ref{ONE} and \ref{G1}, but rather 
to illustrate some of the contexts to which they can be applied.

\begin{corol}
\label{rothberger}

If TWO has a winning strategy in the Rothberger game on a
topological space $X$, then $X$ is productively Rothberger.

\end{corol}

\begin{proof}

Let $Y$ be a Rothberger space. Now apply Proposition
\ref{G1}$(b)$ with $P'=Q'=(X,\tau,\in)$ and
$P=Q=(Y,\rho,\in)$.
\end{proof}

It is worth comparing Corollary \ref{rothberger} with Theorem 11(3) of
\cite{bab.prod}, in which a product of two metric spaces is proven to
be Rothberger under a weaker hypothesis on one of the spaces and a
stronger hypothesis on the other.

The next result was first stated in \cite{telg?}; see also
Theorem 3.1 of \cite{yajima}.\footnote{We thank Piotr Szewczak for
  bringing the paper \cite{yajima} to our attention.}

\begin{corol}[Telg\'arsky \cite{telg?}]
\label{ONE-p-o}

The property ``ONE has a winning strategy in the point-open
game'' (equivalently, ``TWO has a winning strategy in the
Rothberger game'') is preserved under finite products.

\end{corol}

\begin{proof}

Let $X$ and $Y$ be spaces on which ONE has a winning strategy
in the point-open game. Apply Proposition
\ref{ONE}$(c)$ with $P=Q=(X,\tau,\in)$ and
$P'=Q'=(Y,\rho,\in)$.
\end{proof}

The next result is also a consequence of Theorem 3.1 of \cite{yajima}.
By the \emph{compact-open game} on a topological space $X$ we mean the
following game: In each inning $n\in\omega$, ONE chooses a compact
subset $C_n$ of $X$, and then TWO picks an open set $U_n$ with
$C_n\subseteq U_n$; ONE wins if $X=\bigcup_{n\in\omega}U_n$, and loses
otherwise.

\begin{corol}[Yajima \cite{yajima}]
\label{ONE-c-o}

The property ``ONE has a winning strategy in the compact-open
game''
is preserved under finite products.

\end{corol}

\begin{proof}

Let $X$ and $Y$ be spaces on which ONE has a winning strategy
in the compact-open game. Apply Proposition
\ref{ONE}$(c)$ with $P=(K(X),\tau,\subseteq)$, $Q=(X,\tau,\in)$,
$P'=(K(Y),\rho,\subseteq)$ and $Q'=(Y,\rho,\in)$. The result
follows from the observation that, whenever $C_1\in K(X)$, $C_2\in
K(Y)$ and $W$ is an open subset of $X\times Y$ with $C_1\times
C_2\subseteq W$, there exist $U_1\in\tau$ and $U_2\in\rho$ satisfying
$C_1\times C_2\subseteq U_1\times U_2\subseteq W$.
\end{proof}

\begin{corol}
\label{omega-1}

Let $X$ be a topological space on which TWO has a winning
strategy in the game $\mathsf G_1(\Omega,\Omega)$. Then:
\begin{itemize}
\item[$(a)$]
the product $X\times Y$ satisfies $\mathsf{S}_1(\Omega,\Omega)$ for
every topological space $Y$ satisfying $\mathsf{S}_1(\Omega,\Omega)$;
\item[$(a')$]
if $Y$ is a topological space that is Rothberger in every finite
power, then the product $X\times Y$ is Rothberger in every finite
power;
\item[$(b)$]
if TWO has a winning strategy in the game $\mathsf
G_1(\Omega,\Omega)$ on a topological space $Y$, then TWO has
a winning strategy in $\mathsf G_1(\Omega,\Omega)$ on the product
$X\times Y$.
\end{itemize}
\end{corol}

\begin{proof}

By applying Proposition \ref{G1}$(b)$-$(c)$ with
$P'=Q'=([X]^{<\aleph_0},\tau,\subseteq)$ and
$P=Q=([Y]^{<\aleph_0},\rho,\subseteq)$,
we obtain $(a)$ and $(b)$ --- note that every $\omega$-cover
$\mathcal{U}$ of the product $X\times Y$ has an open refinement
$\mathcal{V}$ that is also an $\omega$-cover and such that every
element of $\mathcal{V}$ is a basic open set of the form $U\times V$.
Now $(a')$ follows from the fact that a topological space
satisfies $\mathsf S_1(\Omega,\Omega)$ if and only if all of its
finite powers are Rothberger \cite{sakai}.
\end{proof}

For the next result, recall that a space $X$ is
\emph{weakly Lindel\"of} \cite{frolik} if
$(\mathcal{O}_X,\mathcal{D}_X)$-Lindel\"of holds. Furthermore, $X$ is
\emph{weakly Rothberger} \cite{daniels} if
$\mathsf{S}_1(\mathcal{O}_X,\mathcal{D}_X)$ holds.

\begin{corol}
\label{weakly.rothberger}

Let $X$ be a topological space on which TWO has a winning
strategy in the game $\mathsf G_1(\mathcal O,\mathcal D)$. Then:
\begin{itemize}
\item[$(a)$]
the topological product $X\times Y$ is weakly Lindel\"of whenever
$Y$ is a weakly Lindel\"of space;
\item[$(b)$]
the topological product $X\times Y$ is weakly Rothberger whenever $Y$
is a weakly Rothberger space;
\item[$(c)$]
TWO has a winning strategy in the game $\mathsf G_1(\mathcal
O,\mathcal D)$ on the product $X\times Y$ whenever TWO has a winning
strategy in $\mathsf G_1(\mathcal O,\mathcal D)$ on the topological
space $Y$.\footnote{A result similar to Corollary
\ref{weakly.rothberger}$(c)$ was independently obtained in \cite{3}
under a stronger assumption.}
\end{itemize}
\end{corol}

\begin{proof}

Apply Proposition \ref{G1} with
$P'=(X,\tau,\in)$, $Q'=(\tau\setminus\{\emptyset\},\tau,T)$,
$P=(Y,\rho,\in)$ and $Q=(\rho\setminus\{\emptyset\},\rho,T')$,
where $T=\{(U,V)\in(\tau\setminus\{\emptyset\})\times\tau:U\cap
V\neq\emptyset\}$ and
$T'=\{(U,V)\in(\rho\setminus\{\emptyset\})\times\rho:U\cap
V\neq\emptyset\}$.
\end{proof}


The next corollary gives us an application to a nontopological context
--- namely, variations on the countable chain condition for partial
orders (see e.g. \cite[Chapter III]{kunen} for further details on the
concepts involved in this result).

\begin{corol}
\label{po-ccc}

Let $\mathbb{P}$ be a partial order and
$\mathcal{PD}_{\mathbb{P}}=\{W\subseteq\mathbb{P}:W$ is predense in
$\mathbb{P}\}$. Suppose that TWO has a winning strategy in the
game
$\mathsf{G}_1(\mathcal{PD}_{\mathbb{P}},\mathcal{PD}_{\mathbb{P}})$. Then:
\begin{itemize}
\item[$(a)$]
the partial order $\mathbb{P}\times\mathbb{Q}$ is c.c.c. for every
c.c.c. partial order $\mathbb{Q}$;
\item[$(b)$]
the partial order $\mathbb{P}\times\mathbb{Q}$ satisfies
$\mathsf{S}_1(\mathcal{PD},\mathcal{PD})$ for every partial order
$\mathbb{Q}$ satisfying $\mathsf{S}_1(\mathcal{PD},\mathcal{PD})$.
\end{itemize}
\end{corol}

\begin{proof}

Apply Proposition \ref{G1}$(a)$-$(b)$ with 
$P'=Q'=(\mathbb{P},\mathbb{P},\not\perp)$ and
$P=Q=(\mathbb{Q},\mathbb{Q},\not\perp)$.
The result follows from the observation that a partial order is
c.c.c. if and only if every element of $\mathcal{PD}$ has a countable
subset that is also an element of $\mathcal{PD}$ --- since for every
$X\in\mathcal{PD}$ there is a maximal antichain $A\in\mathcal{PD}$
such that every element of $X$ has an extension in $A$.
\end{proof}

It is worth mentioning that Corollary 1.7 of \cite{open-open} is
equivalent --- by Theorem 4.1 of the same paper --- to the statement
that the property ``TWO has a winning strategy in
$\mathsf G_1(\mathcal{PD},\mathcal{PD})$'' is preserved under
arbitrary products with finite support.

The next result may be viewed as a topological counterpart of
Corollary \ref{po-ccc}.

\begin{corol}
\label{sel-ccc}

Let $X$ be a topological space on which TWO has a winning
strategy in the game $\mathsf G_1(\mathcal D,\mathcal D)$. Then:
\begin{itemize}
\item[$(a)$]
the product $X\times Y$ is c.c.c. for every c.c.c. topological space
$Y$;
\item[$(b)$]
the product $X\times Y$ satisfies
$\mathsf{S}_1(\mathcal{D},\mathcal{D})$ for every topological space
$Y$ that satisfies $\mathsf{S}_1(\mathcal{D},\mathcal{D})$.
\end{itemize}
\end{corol}

\begin{proof}

Apply Corollary \ref{po-ccc} with
$\mathbb{P}=(\tau\setminus\{\emptyset\},\subseteq)$ and
$\mathbb{Q}=(\rho\setminus\{\emptyset\},\subseteq)$.
\end{proof}

Again by Corollary 1.7 of \cite{open-open} (see comment after
Corollary \ref{po-ccc}), the property
``TWO has a winning strategy in $\mathsf
G_1(\mathcal{D},\mathcal{D})$'' is preserved in arbitrary
(Tychonoff) products of topological spaces.

Note that, as a consequence of Corollary \ref{sel-ccc}$(a)$, 
TWO does not have a winning strategy in the game $\mathsf
G_1(\mathcal{D},\mathcal{D})$ played on a Suslin line; therefore,
the game $\mathsf G_1(\mathcal{D},\mathcal{D})$ is undetermined on
Suslin lines --- see
\cite[Theorem 14 and observation following Problem 1]{sch5}.

%

For the next result, we recall that a topological space is
\emph{$R$-separable} \cite{sch6} if it satisfies
$\mathsf{S}_1(\mathfrak{D},\mathfrak{D})$. For a topological space
$X$, we define $\delta(X)=\sup\{d(Z):Z\in\mathfrak{D}_X\}$, where
$d(Z)=\min\{|D|:D\in\mathfrak{D}_Z\}+\aleph_0$.

\begin{corol}
\label{R-sep}

Let $X$ be a topological space such that TWO has a winning
strategy in the game $\mathsf G_1(\mathfrak D,\mathfrak D)$ on
$X$. Then:
\begin{itemize}
\item[$(a)$]
$\delta(X\times Y)=\aleph_0$ for every topological space $Y$ with
  $\delta(Y)=\aleph_0$;
\item[$(b)$]
$X\times Y$ is $R$-separable for every $R$-separable space $Y$;
\item[$(c)$]
TWO has a winning strategy in the game
$\mathsf{G}_1(\mathfrak{D},\mathfrak{D})$ on the product $X\times Y$
for every topological space $Y$ on which TWO has a winning
strategy in the game $\mathsf{G}_1(\mathfrak{D},\mathfrak{D})$.
\end{itemize}
\end{corol}

\begin{proof}

Apply Proposition \ref{G1} with
$P'=Q'=(\tau\setminus\{\emptyset\},X,\ni)$ and
$P=Q=(\rho\setminus\{\emptyset\},Y,\ni)$.
\end{proof}

For our last result in this section, recall that a topological space
$X$ is \emph{countably tight} at a point $x\in X$ if
$\Omega_x$-Lindel\"of holds on $X$, and has
\emph{countable strong fan tightness at $x$} \cite{sakai} if
$\mathsf{S}_1(\Omega_x,\Omega_x)$ holds on $X$.

\begin{corol}
\label{Omega_x}

Let $X$ be a topological space and $x\in X$. Suppose that TWO
has a winning strategy in the game $\mathsf G_1(\Omega_x,\Omega_x)$ on
$X$. Then:
\begin{itemize}
\item[$(a)$]
if a topological space $Y$ is countably tight at a point $y\in Y$,
then $X\times Y$ is countably tight at $(x,y)$;
\item[$(b)$]
if a topological space $Y$ has countable strong fan tightness at
$y\in Y$, then the product space $X\times Y$ has countable strong fan
tightness at the point $(x,y)$;
\item[$(c)$]
if a topological space $Y$ is such that TWO has a winning
strategy in the game $\mathsf G_1(\Omega_y,\Omega_y)$ for a point
$y\in Y$, then TWO has a winning strategy in the game $\mathsf
G_1(\Omega_{(x,y)},\Omega_{(x,y)})$ on the product $X\times Y$.
\end{itemize}
\end{corol}

\begin{proof}

Apply Proposition \ref{G1}$(b)$ with $P'=Q'=(\{U\in\tau:x\in
U\},X,\ni)$ and $P=Q=(\{V\in\rho:y\in
V\},Y,\ni)$.
\end{proof}

We note that Corollary \ref{Omega_x}$(a)$ extends Corollary 2.4 of
\cite{ab}, in which the same conclusion is obtained under the
assumption that $X$ is completely regular. This result will be further improved
in Corollary \ref{answer.aurichi-bella}, in view of the fact that, if
$\pi\mathcal{N}_x^{\aleph_0}$-Lindel\"of holds on $X$, then $X$ is
productively countably tight at $x$ \cite[Corollary 2.3]{bella-sakai}.

\section{Products and finite selections}

By Corollary 3 of \cite{topsoe}, Corollary \ref{ONE-c-o} is equivalent
to the statement that the property ``TWO has a winning strategy
in the Menger game'' is finitely productive in the realm of regular
spaces. We will now give a direct proof\footnote{We could
    put this proof together thanks to an idea due to Leandro Aurichi,
    to whom we are very grateful.}
of this fact without assuming any separation axioms, which will serve
as motivation for Proposition \ref{Gfin}.

\begin{prop}
\label{prod.TWOMenger}

Let $X$ and $Y$ be topological spaces such that TWO has a winning
strategy in both of the games
$\mathsf G_{\mathrm{fin}}(\mathcal O_X,\mathcal O_X)$ and
$\mathsf G_{\mathrm{fin}}(\mathcal O_Y,\mathcal O_Y)$.
Then TWO has a winning strategy in the game $\mathsf
G_{\mathrm{fin}}(\mathcal O_{X\times Y},\mathcal O_{X\times Y})$.

\end{prop}

\begin{proof}

Let
$\varphi:\mbox{}^{<\omega}\mathcal
O_X\setminus\{\emptyset\}\rightarrow[\tau]^{<\aleph_0}\setminus\{\emptyset\}$
and
$\sigma:\mbox{}^{<\omega}\mathcal
O_Y\setminus\{\emptyset\}\rightarrow[\rho]^{<\aleph_0}\setminus\{\emptyset\}$
be winning strategies for TWO in
$\mathsf G_{\mathrm{fin}}(\mathcal O_X,\mathcal O_X)$ and
$\mathsf G_{\mathrm{fin}}(\mathcal O_Y,\mathcal O_Y)$ respectively.
Let $\mathcal
P=\bigcup_{n\in\omega}(\mbox{}^n\omega\times\mbox{}^n\omega)$,
and write
$\omega=\dot\bigcup\{L_s^t:(s,t)\in\mathcal P\}$
with
\begin{itemize}
\item[$\cdot$]
$|L^t_s|=\aleph_0$ for every $(s,t)\in\mathcal P$; and
\item[$\cdot$]
$\min L^t_s>\max(\mathrm{im}(s))$ for every
$(s,t)\in\mathcal P\setminus\{(\emptyset,\emptyset)\}$.
\end{itemize}
Enumerate each set $L_s^t$ as $L_s^t=\{l^t_s(j):j\in\omega\}$,
where $l^t_s(j)<l^t_s(j+1)$ for all $j\in\omega$.

In order to construct a winning strategy for TWO in the game $\mathsf
G_{\mathrm{fin}}(\mathcal O_{X\times Y},\mathcal O_{X\times Y})$,
we may assume that, in each inning $n\in\omega$ of this game, ONE
plays an open cover $\mathcal U_n$ of $X\times Y$ of the form
$\mathcal U_n=\{U^n_i\times V^n_i:i\in I_n\}$. Given such
an open cover, for each $x\in X$, define
$I_n^x=\{i\in I_n:x\in U^n_i\}$ and
$\mathcal V^x_n=\{V^n_i:i\in I^x_n\}\in\mathcal O_Y$.
In each inning $n\in\omega$, we will make use of the
strategy $\sigma$ to assign to each $x\in X$ a finite nonempty subset
of $\mathcal V^x_n$ --- that is, a finite nonempty subset
$F^x_n$ of $I^x_n$. We will then define
$W^x_n=\bigcap_{i\in F^x_n}U^n_i$ for each $x\in
X$, and then make use of the strategy $\varphi$ to choose finitely
many elements of the open cover $\mathcal W_n=
\{W^x_n:x\in X\}$ of $X$ --- that is, a finite subset
$\{x^n_k:k\in\omega\}$ of $X$, here enumerated with infinite
repetition of all of the terms. TWO's answer to $\mathcal U_n$ will
then be the finite set
$\mathcal F_n=\bigcup_{k\in\omega}\left\{U^n_i\times V^n_i:i\in F^{x^n_k}_n\right\}$.
In order to make sure that this will define a winning strategy, when
making use of $\sigma$ we will consider not all of the previous
innings, but only those listed (in a sense that will become clear in
the next paragraph) by the sequences
$s,t\in\mbox{}^{<\omega}\omega$ such that $n\in L_s^t$; also,
when making use of $\varphi$ we will consider a play of $\mathsf
G_{\mathrm{fin}}(\mathcal O_X,\mathcal O_X)$ whose history is given
not by all of the innings $0,1,\dots,n-1$, but only those in
$L^t_s\cap n$.

Now for the details of the procedure.
Suppose that, in the inning $n\in\omega$ of $\mathsf
G_{\mathrm{fin}}(\mathcal O_{X\times Y},\mathcal O_{X\times Y})$, the
play so far is
$(\mathcal U_0,\mathcal F_0,\mathcal U_1,\mathcal F_1,\dots,\mathcal
U_n)$, where each $\mathcal F_m$ is of the form
$\mathcal F_m=\bigcup_{k\in\omega}\left\{U^m_i\times V^m_i:
i\in F^{x^m_k}_m\right\}$, as described in the previous paragraph.
Let $(s,t)\in\mathcal P$ and $j\in\omega$ be such that $n=l^t_s(j)\in
L_s^t$. For each $x\in X$, let
$F^x_n\in[I^x_n]^{<\aleph_0}\setminus\{\emptyset\}$ be such that
$\sigma\left(\left(\mathcal V^{x^{s(r)}_{t(r)}}_{s(r)}
\right)_{r\in\mathrm{dom}(s)}\mbox{}^\frown
\left(\mathcal V_n^x\right)\right)=\{V^n_i:i\in F^x_n\}$.
We can now define the open neighborhood
$W^x_n=\bigcap_{i\in F^x_n}U^n_i$ for each $x\in X$, and then consider
$\mathcal W_n=\{W^x_n:x\in X\}\in\mathcal O_X$.
Now let $\{x^n_k:k\in\omega\}$ be a finite subset of $X$, enumerated
with infinite repetition of all of the terms, satisfying
$\varphi\left(\left(\mathcal W_{l^t_s(h)}
\right)_{h\le j}\right)=\left\{W_n^{x^n_k}:k\in\omega\right\}$.
TWO's answer in the $n$-th inning of the play $(\mathcal U_0,\mathcal
F_0,\mathcal U_1,\mathcal F_1,\dots,\mathcal U_n)$ of $\mathsf
G_{\mathrm{fin}}(\mathcal O_{X\times Y},\mathcal O_{X\times Y})$ is
then the finite subset $\mathcal
F_n=\bigcup_{k\in\omega}\left\{U^n_i\times V^n_i:i\in
F^{x^n_k}_n\right\}$ of $\mathcal U_n$.

Let us now prove that this defines a winning strategy for TWO in
$\mathsf G_{\mathrm{fin}}(\mathcal O_{X\times Y},\mathcal O_{X\times Y})$.
Let $(x,y)\in X\times Y$ be arbitrary.
Pick $n_0\in L^\emptyset_\emptyset$ with
$x\in\bigcup_{k\in\omega}W_{n_0}^{x^{n_0}_k}$ --- such an $n_0$ must
exist, for otherwise
\begin{center}
{\renewcommand{\arraystretch}{3.0}
\renewcommand{\tabcolsep}{0.25cm}
\begin{tabular}{c|c}
ONE & TWO \\

\hline

$\mathcal W_{l^\emptyset_\emptyset(0)}$ &
$\left\{W_{l^\emptyset_\emptyset(0)}^{x^{l^\emptyset_\emptyset(0)}_k}:k\in\omega\right\}$\\
$\mathcal W_{l^\emptyset_\emptyset(1)}$ &
$\left\{W_{l^\emptyset_\emptyset(1)}^{x^{l^\emptyset_\emptyset(1)}_k}:k\in\omega\right\}$\\
$\vdots$ & $\vdots$ \\
$\mathcal W_{l^\emptyset_\emptyset(j)}$ &
$\left\{W_{l^\emptyset_\emptyset(j)}^{x^{l^\emptyset_\emptyset(j)}_k}:k\in\omega\right\}$\\
$\vdots$ & $\vdots$
\end{tabular}}\\
\end{center}
would be a play of $\mathsf G_{\mathrm{fin}}(\mathcal O_X,\mathcal
O_X)$ in which TWO plays according to the winning strategy $\varphi$
and loses. Let then $k_0\in\omega$ be such that $x\in
W_{n_0}^{x^{n_0}_{k_0}}$.
Similarly, we can recursively pick $n_r\in
L^{(k_l)_{l<r}}_{(n_l)_{l<r}}$ and $k_r\in\omega$ satisfying $x\in
W_{n_r}^{x^{n_r}_{k_r}}$ for each $r\in\omega$, thus defining the
sequences
$(n_r)_{r\in\omega}\in\mbox{}^\omega\omega$ and
$(k_r)_{r\in\omega}\in\mbox{}^\omega\omega$.
Now
\begin{center}
{\renewcommand{\arraystretch}{2.5}
\renewcommand{\tabcolsep}{0.25cm}
\begin{tabular}{c|c}
ONE & TWO \\

\hline

$\mathcal V^{x_{k_0}^{n_0}}_{n_0}$ &
$\left\{V^{n_0}_i:i\in F_{n_0}^{x^{n_0}_{k_0}}\right\}$\\
$\mathcal V^{x_{k_1}^{n_1}}_{n_1}$ &
$\left\{V^{n_1}_i:i\in F_{n_1}^{x^{n_1}_{k_1}}\right\}$\\
$\vdots$ & $\vdots$\\
$\mathcal V^{x_{k_r}^{n_r}}_{n_r}$ &
$\left\{V^{n_r}_i:i\in F_{n_r}^{x^{n_r}_{k_r}}\right\}$\\
$\vdots$ & $\vdots$
\end{tabular}}\\
\end{center}
is a play of $\mathsf G_{\mathrm{fin}}(\mathcal O_Y,\mathcal O_Y)$ in
which TWO follows the winning strategy $\sigma$; therefore, there is
$h\in\omega$ such that, for some $i'\in F_{n_r}^{x^{n_r}_{k_r}}$, we
have $y\in V^{n_r}_{i'}$.
Since
$x\in W_{n_r}^{x^{n_r}_{k_r}}=
\bigcap_{i\in F^{x^{n_r}_{k_r}}_{n_r}}U^{n_r}_i\subseteq U^{n_r}_{i'}$
also holds, it follows that $(x,y)\in U^{n_r}_{i'}\times
V^{n_r}_{i'}$. Thus, the sets played by TWO cover the product
$X\times Y$.
\end{proof}

In order to capture the main aspect needed to adapt the proof of
Proposition \ref{prod.TWOMenger} to the general setting of relations,
we define the following auxiliary concept.

%
%

\begin{defin}
\label{def.order}

Let $P=(A,B,R)$ be a relation and $\preceq$ be a partial order on the
set $B$. We say that $\preceq$ is
\begin{itemize}
\item[$\cdot$]
\emph{downwards $P$-compatible} if,
for all $a\in A$ and $b_1,b_2\in B$,
$$(aRb_1\;\&\; aRb_2)\rightarrow\exists \tilde{b}\in B\;(aR\tilde{b}\;\&\; \tilde{b}\preceq
  b_1\;\&\; \tilde{b}\preceq b_2);$$
\item[$\cdot$]
\emph{upwards $P$-compatible} if,
for all $a\in A$ and $b_1,b_2\in B$,
$$(aRb_1\;\&\; b_1\preceq b_2)\rightarrow aRb_2.$$
\end{itemize}
\end{defin}

We can now state the main result of this section.

\begin{prop}
\label{Gfin}

Let $P=(A,B,R)$, $P'=(A',B',R')$, $Q=(C,D,T)$
and $Q'=(C',D',T')$ be relations with $B=D$
such that there is a partial order $\preceq$ on $B$ that is both
downwards $P$-compatible and upwards $Q$-compatible.
Suppose that TWO has a winning strategy in the game
$\mathsf{G}_{\mathrm{fin}}(\mathrm{Dom}(P'),\mathrm{Dom}(Q'))$.

\begin{itemize}

\item[$(a)$]
If $(\mathrm{Dom}(P),\mathrm{Dom}(Q))$-Lindel\"of holds, then
$(\mathrm{Dom}(P\otimes P'),\mathrm{Dom}(Q\otimes Q'))$-Lindel\"of
holds.

\item[$(b)$]
If $\mathsf{S}_{\mathrm{fin}}(\mathrm{Dom}(P),\mathrm{Dom}(Q))$ holds,
then $\mathsf{S}_{\mathrm{fin}}(\mathrm{Dom}(P\otimes
P'),\mathrm{Dom}(Q\otimes Q'))$ also holds.

\item[$(c)$]
If TWO has a winning strategy in the game
$\mathsf{G}_{\mathrm{fin}}(\mathrm{Dom}(P),\mathrm{Dom}(Q))$, then
TWO has a winning strategy in the game
$\mathsf{G}_{\mathrm{fin}}(\mathrm{Dom}(P\otimes
P'),\mathrm{Dom}(Q\otimes Q'))$.

\end{itemize}

\end{prop}

\begin{proof}

Let
$\sigma:\mbox{}^{<\omega}\mathrm{Dom}(P')\setminus\{\emptyset\}\rightarrow[B']^{<\aleph_0}\setminus\{\emptyset\}$
be a winning strategy for TWO in the game
$\mathsf{G}_{\mathrm{fin}}(\mathrm{Dom}(P'),\mathrm{Dom}(Q'))$.\\

$(a)$
Let $\{(b_i,b'_i):i\in I\}\in\mathrm{Dom}(P\otimes P')$ be fixed.
For each $x\in A$, define $Z_x=\{b'_i:i\in I_x\}$, where 
$I_x=\{i\in I:xRb_i\}$; note that $Z_x\in\mathrm{Dom}(P')$.

We will construct indexed families
$\langle F^s_x:s\in\mbox{}^{<\omega}\omega,x\in A\rangle$ and
$\langle a^s_n:s\in\mbox{}^{<\omega}\omega,n\in\omega\rangle$ with
\begin{itemize}
\item[$\cdot$]
$F^s_x\in[I_x]^{<\aleph_0}\setminus\{\emptyset\}$ for all
$s\in\mbox{}^{<\omega}\omega$ and $x\in A$; and
\item[$\cdot$]
$a^s_n\in A$ for all $s\in\mbox{}^{<\omega}\omega$ and $n\in\omega$.
\end{itemize}
We proceed recursively as follows.

Suppose that $k\in\omega$ is such that $\langle
a^t_n:t\in\mbox{}^{<k}\omega,n\in\omega\rangle$ and $\langle
F^t_x:t\in\mbox{}^{<k}\omega,x\in A\rangle$ have already been
constructed, and let $s\in\mbox{}^k\omega$ be fixed.
For each $x\in A$, let $F^s_x\in[I_x]^{<\aleph_0}\setminus\{\emptyset\}$
be such that
$$
\sigma\left(\left(
Z_{a^{s\upharpoonright j}_{s(j)}}
\right)_{j<k}\mbox{}^\frown
\left(
Z_x
\right)
\right)=\{b'_i:i\in F^s_x\}.
$$
Since $\preceq$ is downwards $P$-compatible, for each $x\in A$ we may
then choose $b^s_x\in B$ satisfying $xRb^s_x$ and $b^s_x\preceq b_i$
for all $i\in F^s_x$, thus obtaining a set $\{b^s_x:x\in
A\}\in\mathrm{Dom}(P)$. By our hypothesis, this set has a countable
subset that is an element of $\mathrm{Dom}(Q)$; let then
$(a^s_n)_{n\in\omega}$ be a sequence of elements of $A$ with
$\left\{b^s_{a^s_n}:n\in\omega\right\}\in\mathrm{Dom}(Q)$. This
concludes the recursive construction.

The proof will be finished once we show that
$$
\bigcup_{s\in\mbox{}^{<\omega}\omega}\bigcup_{n\in\omega}\{(b_i,b'_i):i\in
F^s_{a^s_n}\}\in\mathrm{Dom}(Q\otimes Q').
$$
To this end, let $(c,c')\in C\times C'$ be arbitrary. Pick
recursively, for each $k\in\omega$, an $n_k\in\omega$ such that
$cTb^{(n_j)_{j<k}}_{a^{(n_j)_{j<k}}_{n_k}}$ --- this is possible since
$\left\{b^{(n_j)_{j<k}}_{a^{(n_j)_{j<k}}_n}:n\in\omega\right\}\in\mathrm{Dom}(Q)$.
It follows that
\begin{center}
{\renewcommand{\arraystretch}{2.5}
\renewcommand{\tabcolsep}{0.25cm}
\begin{tabular}{c|c}
ONE & TWO \\

\hline

$Z_{a^\emptyset_{n_0}}$ &
$\left\{b'_i:i\in F^\emptyset_{a^\emptyset_{n_0}}\right\}$\\
$Z_{a^{(n_0)}_{n_1}}$ &
$\left\{b'_i:i\in F^{(n_0)}_{a^{(n_0)}_{n_1}}\right\}$\\
$Z_{a^{(n_0,n_1)}_{n_2}}$ &
$\left\{b'_i:i\in F^{(n_0,n_1)}_{a^{(n_0,n_1)}_{n_2}}\right\}$\\
$\vdots$ & $\vdots$\\
$Z_{a^{(n_j)_{j<k}}_{n_k}}$ &
$\left\{b'_i:i\in F^{(n_j)_{j<k}}_{a^{(n_j)_{j<k}}_{n_k}}\right\}$\\
$\vdots$ & $\vdots$
\end{tabular}}\\
\end{center}
is a play of
$\mathsf{G}_{\mathrm{fin}}(\mathrm{Dom}(P'),\mathrm{Dom}(Q'))$ in
which TWO makes use of the winning strategy $\sigma$; hence,
for some $k\in\omega$, there is $i\in
F^{(n_j)_{j<k}}_{a^{(n_j)_{j<k}}_{n_k}}$ such that $c'T'b'_{i}$.
Since $\preceq$ is upwards $Q$-compatible and
$cTb^{(n_j)_{j<k}}_{a^{(n_j)_{j<k}}_{n_k}}$, we also have
$cTb_{i}$;
thus, $(c,c')T\otimes T'(b_{i},b'_{i})$, as required.

$(b)$
Let $(Z_n)_{n\in\omega}$ be a sequence of elements of
$\mathrm{Dom}(P\otimes P')$. For each $n\in\omega$, write
$Z_n=\{(b^n_i,b'^n_i):i\in I_n\}$ and, for each
$x\in A$, define $I_n^x=\{i\in I_n:x R
b^n_i\}$ and $Z_n^x=\{b'^n_i:i\in
I_n^x\}\in\mathrm{Dom}(P')$.

Let
$\mathcal{P}=\bigcup_{j\in\omega}(\mbox{}^j\omega\times\mbox{}^j\omega)$.
Fix a partition
$\omega=\dot\bigcup_{(s,t)\in\mathcal{P}}L_{s,t}$
with $|L_{s,t}|=\aleph_0$ for all $(s,t)\in\mathcal{P}$, and write
$L_{s,t}=\{l^{s,t}_{m}:m\in\omega\}$ with $l^{s,t}_{m_1}\neq
l^{s,t}_{m_2}$ if $m_1\neq m_2$.
We will construct an indexed family $\langle
a_{s,t}:(s,t)\in\mathcal{P}\setminus\{(\emptyset,\emptyset)\}\rangle$
of elements of $A$ as follows.

Suppose that $h\in\omega$ is such that $a_{s,t}$ has already been
defined for all
$(s,t)\in\bigcup_{0<j\le h}(\mbox{}^j\omega\times\mbox{}^j\omega)$,
and let $(\tilde{s},\tilde{t})\in\mbox{}^h\omega\times\mbox{}^h\omega$
be fixed. For each $x\in A$ and each $m\in\omega$, let
$F^{\tilde{s},\tilde{t}}_m(x)\in\left[I^x_{l^{\tilde{s},\tilde{t}}_m}\right]^{<\aleph_0}\setminus\{\emptyset\}$
be such that
$$
\sigma\left(\left(Z_{l^{\tilde{s}\upharpoonright k,\tilde{t}\upharpoonright k}_{\tilde{t}(k)}}
^{a_{\tilde{s}\upharpoonright(k+1),\tilde{t}\upharpoonright(k+1)}}\right)_{k<h}
\mbox{}^\frown\left(Z^x_{l^{\tilde{s},\tilde{t}}_m}
\right)\right)=
\left\{b'^{l^{\tilde{s},\tilde{t}}_m}_i:i\in F^{\tilde{s},\tilde{t}}_m(x)\right\};
$$
then, making use of the fact that $\preceq$ is downwards
$P$-compatible, pick $b^{\tilde{s},\tilde{t}}_m(x)\in B$ satisfying
$xRb^{\tilde{s},\tilde{t}}_m(x)$ and
$b^{\tilde{s},\tilde{t}}_m(x)\preceq b^{l^{\tilde{s},\tilde{t}}_m}_i$
for all $i\in F^{\tilde{s},\tilde{t}}_m(x)$. Since
$\{b^{\tilde{s},\tilde{t}}_m(x):x\in A\}\in\mathrm{Dom}(P)$ for each
$m\in\omega$, we can apply
$\mathsf{S}_{\mathrm{fin}}(\mathrm{Dom}(P),\mathrm{Dom}(Q))$ and
thus obtain a sequence
$(\mathcal{F}^{\tilde{s},\tilde{t}}_m)_{m\in\omega}$ of finite
nonempty subsets of $A$ satisfying
$\bigcup_{m\in\omega}\{b^{\tilde{s},\tilde{t}}_m(x):x\in\mathcal{F}^{\tilde{s},\tilde{t}}_m\}\in\mathrm{Dom}(Q)$.
Now, for each $m\in\omega$, write
$\mathcal{F}^{\tilde{s},\tilde{t}}_m=\{a_{\tilde{s}^\smallfrown(m),\tilde{t}^\smallfrown(n)}:n\in\omega\}$
with each element of $\mathcal{F}^{\tilde{s},\tilde{t}}_m$ appearing
infinitely many times in the listing.

This completes the definition of $\langle
a_{s,t}:(s,t)\in\mathcal{P}\setminus\{(\emptyset,\emptyset)\}\rangle$.
We now claim that the family
$$\bigcup_{(s,t)\in\mathcal{P}}\bigcup_{m,n\in\omega}\left\{\left(b^{l^{s,t}_m}_i,b'^{l^{s,t}_m}_i\right):i\in
F^{s,t}_m(a_{s^\smallfrown(m),t^\smallfrown(n)})\right\}$$
is an element of $\mathrm{Dom}(Q\otimes Q')$.

Let then $(c,c')\in C\times C'$ be arbitrary.
Recursively for each $h\in\omega$, choose $m_h,n_h\in\omega$ such that
$cTb^{(m_k)_{k<h},(n_k)_{k<h}}_{m_h}
(a_{(m_k)_{k\le h},(n_k)_{k\le h}})$.
Since $\preceq$ is upwards $Q$-compatible,
this implies that
$cTb^{l^{(m_k)_{k<h},(n_k)_{k<h}}_{m_h}}_i$ for all $i\in
F^{(m_k)_{k<h},(n_k)_{k<h}}_{m_h}(a_{(m_k)_{k\le h},(n_k)_{k\le h}})$; hence,
it suffices to verify that, for
some $h\in\omega$, there is $i\in
F^{(m_k)_{k<h},(n_k)_{k<h}}_{m_h}(a_{(m_k)_{k\le h},(n_k)_{k\le h}})$
with $c'T'b'^{l^{(m_k)_{k<h},(n_k)_{k<h}}_{m_h}}_i$. Indeed, if this
were not the case, 
\begin{center}
{\renewcommand{\arraystretch}{2.5}
\renewcommand{\tabcolsep}{0.25cm}
\begin{tabular}{c|c}
ONE & TWO \\

\hline

$Z_{l^{\emptyset,\emptyset}_{n_0}}^{a_{(m_0),(n_0)}}$ &
$\left\{b'^{l^{\emptyset,\emptyset}_{m_0}}_i:i\in F^{\emptyset,\emptyset}_{m_0}(a_{(m_0),(n_0)})\right\}$\\
$Z_{l^{(m_0),(n_0)}_{n_1}}^{a_{(m_0,m_1),(n_0,n_1)}}$ &
$\left\{b'^{l^{(m_0),(n_0)}_{m_1}}_i:i\in F^{(m_0),(n_0)}_{m_1}(a_{(m_0,m_1),(n_0,n_1)})\right\}$\\
$\vdots$ & $\vdots$\\
$Z_{l^{(m_k)_{k<h},(n_k)_{k<h}}_{n_h}}^{a_{(m_k)_{k\le h},(n_k)_{k\le h}}}$ &
$\left\{b'^{l^{(m_k)_{k<h},(n_k)_{k<h}}_{m_h}}_i:i\in
F^{(m_k)_{k<h},(n_k)_{k<h}}_{m_h}(a_{(m_k)_{k\le h},(n_k)_{k\le h}})\right\}$\\
$\vdots$ & $\vdots$
\end{tabular}}\\
\end{center}
would be a play of
$\mathsf{G}_{\mathrm{fin}}(\mathrm{Dom}(P'),\mathrm{Dom}(Q'))$ in
which TWO follows the winning strategy $\sigma$ and loses.\\

$(c)$ Analogous to the proof of Proposition \ref{prod.TWOMenger}.
\end{proof}

The following consequence of Proposition \ref{Gfin} was originally
proven by Telg\'arsky under the extra assumption that the space is
regular; it can be obtained by putting together Corollary 14.14 and
Theorem 14.12 of \cite{telg1} and Corollary 3 of \cite{topsoe}.
Here we follow the standard terminology and refer to the game
$\mathsf{G}_{\mathrm{fin}}(\mathcal{O},\mathcal{O})$ as
\emph{the Menger game}; a topological space is \emph{Menger}
\cite{hur} if it satisfies
$\mathsf{S}_{\mathrm{fin}}(\mathcal{O},\mathcal{O})$.

\begin{corol}[Telg\'arsky \cite{telg1,topsoe}, for regular spaces]
\label{menger}

If TWO has a winning strategy in the Menger game on a topological
space $X$, then $X$ is both productively Lindel\"of and productively
Menger.

\end{corol}

\begin{proof}

Let $Y$ be a Lindel\"of (respectively, Menger) space. Apply
Proposition \ref{Gfin}$(a)$ (respectively, $(b)$) with
$P'=Q'=(X,\tau,\in)$, $P=Q=(Y,\rho,\in)$ and
$U\preceq V\leftrightarrow U\subseteq V$.
\end{proof}

As in the observation made after Corollary \ref{rothberger}, it is
worth comparing Corollary \ref{menger} with Theorem 11(2) of
\cite{bab.prod}.

It is also worth remarking that, assuming the Continuum Hypothesis,
one can construct a Sierpi\'nski set $S\subseteq\mathbb{R}$ such that
$\mathbb{R}\setminus\mathbb{Q}$ is a continuous image of $S\times S$
\cite[Theorem 2.12]{just}, which implies that $S\times S$ is not
Menger. By Proposition 14 of \cite{bab.prod}, TWO has a winning
strategy in the ``length-($\omega+\omega$)'' variation of the Menger
game played on every Sierpi\'nski set. This shows that the existence
of a winning strategy for TWO in this longer version of the Menger
game, although stronger than the Menger property, is not strong enough
to imply its productivity.

\begin{corol}
\label{omega-fin}

Let $X$ be a topological space on which TWO has a winning
strategy in the game $\mathsf G_{\mathrm{fin}}(\Omega,\Omega)$. Then:
\begin{itemize}
\item[$(a)$]
if $Y$ is a topological space such that $\Omega_Y$-Lindel\"of holds,
then $\Omega_{X\times Y}$-Lindel\"of holds;
\item[$(a')$]
if $Y$ is a topological space such that every finite power of $Y$ is
Lindel\"of, then every finite power of $X\times Y$ is Lindel\"of;
\item[$(b)$]
the product $X\times Y$ satisfies
$\mathsf{S}_{\mathrm{fin}}(\Omega,\Omega)$ for every topological space
$Y$ satisfying $\mathsf{S}_{\mathrm{fin}}(\Omega,\Omega)$;
\item[$(b')$]
if $Y$ is a topological space such that every finite power of $Y$ is
Menger, then every finite power of $X\times Y$ is Menger;
\item[$(c)$]
if TWO has a winning strategy in the game $\mathsf
G_{\mathrm{fin}}(\Omega,\Omega)$ on a topological space $Y$, then he
has a winning strategy in $\mathsf G_{\mathrm{fin}}(\Omega,\Omega)$ on
the product $X\times Y$.
\end{itemize}
\end{corol}

\begin{proof}

First, we recall that a topological space $Y$ is Lindel\"of in every
finite power if and only if
$\Omega_Y$-Lindel\"of holds \cite{g-n},
and that it is Menger in every finite power if and only if
$\mathsf S_{\mathrm{fin}}(\Omega_Y,\Omega_Y)$ holds
\cite[Theorem 3.9]{just}.
Now apply Proposition \ref{Gfin} with
$P'=Q'=([X]^{<\aleph_0},\tau,\subseteq)$,
$P=Q=([Y]^{<\aleph_0},\rho,\subseteq)$ and
$U\preceq V\leftrightarrow U\subseteq V$, having in mind the
observation made in the proof of Corollary \ref{omega-1}.
\end{proof}

For a topological space $X$, the space of all real-valued continuous
functions on $X$ regarded as a subspace of the (Tychonoff)
power $\mathbb{R}^X$ is denoted by $C_p(X)$.
Recall that a topological space has \emph{countable fan tightness}
\cite{arh2} at a point $x$ if
$\mathsf{S}_{\mathrm{fin}}(\Omega_x,\Omega_x)$ holds.

\begin{corol}
\label{Omega_o-fin}

Let $X$ and $Y$ be completely regular spaces. Assume that TWO has a
winning strategy in the game
$\mathsf{G}_{\mathrm{fin}}(\Omega_{\mathbf{0}},\Omega_{\mathbf{0}})$
played on $C_p(X)$.
\begin{itemize}
\item[$(a)$]
If $C_p(Y)$ is countably tight, then $C_p(X)\times C_p(Y)$
is also countably tight.
\item[$(b)$]
If $C_p(Y)$ has countable fan tightness, then the product
$C_p(X)\times C_p(Y)$ also has countable fan tightness.
\item[$(c)$]
If TWO has a winning strategy in
$\mathsf{G}_{\mathrm{fin}}(\Omega_{\mathbf{0}},\Omega_{\mathbf{0}})$
played on $C_p(Y)$, then TWO has a winning strategy in
$\mathsf{G}_{\mathrm{fin}}(\Omega_{(\mathbf{0},\mathbf{0})},\Omega_{(\mathbf{0},\mathbf{0})})$
played on the product $C_p(X)\times C_p(Y)$.
\end{itemize}
\end{corol}

\begin{proof}

This follows from Corollary \ref{omega-fin}, in view of the following
results:
\begin{itemize}
\item[$\cdot$]
for a completely regular space $X$, TWO has a winning strategy in
$\mathsf{G}_{\mathrm{fin}}(\Omega,\Omega)$ on $X$ if and only if TWO
has a winning strategy in
$\mathsf{G}_{\mathrm{fin}}(\Omega_{\mathbf{0}},\Omega_{\mathbf{0}})$
on $C_p(X)$ \cite[Theorem 26]{sch-tight};
\item[$\cdot$]
a completely regular space $X$ is Menger in every finite power if and only if
$C_p(X)$ has countable fan tightness \cite[Theorem 4]{arh2};
\item[$\cdot$]
a completely regular space $X$ is Lindel\"of in every finite power if and only
if $C_p(X)$ is countably tight --- see \cite[Theorem 2]{arh1} and
\cite[Theorem 1]{pyt}.
\end{itemize}
\end{proof}

The hypothesis about the partial order $\preceq$ in Proposition
\ref{Gfin} is essential and there is no possibility of obtaining a
result as general as Proposition \ref{G1} for properties involving
finite selections. This can be seen, for example, by considering the
same properties appearing in Corollary \ref{Omega_o-fin}: It follows
from a result of Uspenski\u\i\mbox{} \cite[Theorem 1]{uspenskii} that
a space $X$ is Lindel\"of in the $G_\delta$-topology is and only if
$C_p(X)$ is productively countably tight; note that, by Theorem 26 of
\cite{sch-tight}, TWO has a winning strategy in
$\mathsf{G}_{\mathrm{fin}}(\Omega_{\mathbf{0}},\Omega_{\mathbf{0}})$
on $C_p(\mathbb{R})$ since TWO has a winning strategy in
$\mathsf{G}_{\mathrm{fin}}(\Omega,\Omega)$ on $\mathbb{R}$
\cite[Lemma 2]{bella-SS}, and yet $C_p(\mathbb{R})$ is not
productively countably tight in view of Uspenski\u\i\mbox{}'s result.

\section{$\gamma$-dominating sequences}

In this section we study games defined in terms of
\emph{$\gamma$-dominating sequences}, which allow us
to express convergence-like properties.

We start off with a simple fact:

\begin{lemma}
\label{equiv.gamma}

Let $P=(A,B,R)$ be a relation and
$(z_n)_{n\in\omega}\in\mbox{}^\omega B$. The following assertions
are equivalent:

\begin{itemize}

\item[$(a)$]
$\{n\in\omega:\neg (aRz_n)\}$ is finite for all $a\in A$;

\item[$(b)$]
$\{z_n:n\in X\}\in\mathrm{Dom}(P)$ for all $X\in[\omega]^{\aleph_0}$.

\end{itemize}

\end{lemma}

\begin{defin}
\label{def.gamma}

Let $P=(A,B,R)$ be a relation.
A sequence $(z_n)_{n\in\omega}\in\mbox{}^\omega B$ is
\emph{$\gamma$-dominating in $P$} if the conditions in Lemma
\ref{equiv.gamma} hold. The family of all $\gamma$-dominating
sequences in $P$ will be denoted by $\mathrm{Dom}_\gamma(P)$.

\end{defin}

We shall now consider variations of the games $\mathsf{G}$ and
$\mathsf{G}_1$ involving $\gamma$-dominating sequences.

\begin{defin}
\label{def.G_gamma}

Let $P$ and $Q$ be relations.
\begin{itemize}
\item[$(a)$]
The game $\mathsf{G}_\gamma(P,Q)$ is
played according to the same rules as $\mathsf{G}(P,Q)$, but now
ONE wins if $(b_n)_{n\in\omega}\in\mathrm{Dom}_\gamma(Q)$, and
loses otherwise.
\item[$(b)$]
In a slight abuse of notation (cf. Definition \ref{def.G1}), we shall
designate by $\mathsf{G}_1(\mathrm{Dom}(P),\mathrm{Dom}_\gamma(Q))$
the game played according to the same rules as
$\mathsf{G}_1(\mathrm{Dom}(P),\mathrm{Dom}(Q))$, but in which the
winner is TWO if $(b_n)_{n\in\omega}\in\mathrm{Dom}_\gamma(Q)$ and
ONE otherwise.
\end{itemize}
\end{defin}

The games defined above satisfy the following duality theorem,
which parallels Theorem \ref{dual}.

\begin{thm}[Galvin \cite{galvin}]
\label{dual-gamma}

The games $\mathsf{G}_\gamma(P,Q)$ and
$\mathsf{G}_1(\mathrm{Dom}(P),\mathrm{Dom}_\gamma(Q))$ are dual for
all relations $P$ and $Q$.

\end{thm}

The following consequence of Theorem \ref{dual-gamma} was observed in
the proof of Proposition 1 of \cite{gerlits2}.
First, let us recall
the game $\mathsf{G}_{O,P}(X,x)$, introduced by Gruenhage in
\cite{gruen76}. Let a topological space $X$ and $x\in X$ be fixed. In
each inning $n\in\omega$, ONE picks $V_n\in\tau_x$, and then TWO
chooses $x_n\in V_n$. ONE wins if the sequence $(x_n)_{n\in\omega}$
converges to $x$, and loses otherwise.

\begin{corol}[Gerlits \cite{gerlits2}]
\label{dual.o-p_x}

Let $X$ be a topological space, $x\in X$
and $\Gamma_x=\{(x_n)_{n\in\omega}\in\mbox{}^\omega
X:x_n\stackrel{n\rightarrow\infty}\longrightarrow x\}$.
Then the game $\mathsf{G}_{O,P}(X,x)$ and the game
$\mathsf{G}_1(\Omega_x,\Gamma_x)$ on $X$ are dual.

\end{corol}

\begin{proof}

Apply Theorem \ref{dual-gamma} with $P=Q=(\tau_x,X,\ni)$.
\end{proof}

Our main goal in this section is to find conditions under which
the existence of a winning strategy for ONE in the game $\mathsf{G}(P,Q)$
yields the existence of a winning strategy for ONE also in the game
$\mathsf{G}_\gamma(P,Q)$. In order to formulate such conditions, we will
need the following auxiliary notion.

\begin{defin}[cf. Definition \ref{def.order}]
\label{down.P-compat}

Let $P=(A,B,R)$ be a relation. We say that a partial order
$\trianglelefteq$ on $A$ is
%
%
%
\emph{downwards $P$-compatible} if, for
all $a_1,a_2\in A$ and $b\in B$,
$$
(a_1\trianglelefteq a_2\;\&\;a_2Rb)\rightarrow a_1Rb.
$$


\end{defin}

The argument for the next result is essentially taken from
Theorem 3.9 of \cite{gruen76} (which we state as
Corollary \ref{equiv.W}); see also Theorem 1 of \cite{g-n} (also
stated here as Corollary \ref{equiv-g-n}).

\begin{thm}[Gruenhage \cite{gruen76}]
\label{gn}

Let $P=(A,B,R)$ and $Q=(C,D,T)$ be relations. Suppose that there is a
downwards $P$-compatible partial order $\trianglelefteq$ on $A$ such
that, for each finite subset $F$ of $A$, there is $\tilde{a}(F)\in A$
satisfying $a\trianglelefteq \tilde{a}(F)$ for all $a\in F$. Then the
following conditions are equivalent:

\begin{itemize}

\item[$(a)$]
ONE has a winning strategy in $\mathsf{G}(P,Q)$;

\item[$(b)$]
ONE has a winning strategy in $\mathsf{G}_\gamma(P,Q)$.

\end{itemize}

\end{thm}

\begin{proof}

The implication $(b)\rightarrow(a)$ is immediate. We will prove
that $(a)$ implies $(b)$.

Let $\sigma:\mbox{}^{<\omega}B\rightarrow A$ be a winning strategy for
ONE in $\mathsf{G}(P,Q)$.
For each $n\in\omega$, let $S_n$ be the (finite) set of all strictly
increasing sequences with range included in $n$. Now define
$\varphi:\mbox{}^{<\omega}B\rightarrow A$ by
$\varphi((b_j)_{j<n})=\tilde{a}({\{\sigma((b_{s(i)})_{i\in\mathrm{dom}(s)}):s\in S_n\}})$.
We claim that $\varphi$ is a winning strategy for ONE in
$\mathsf{G}_\gamma(P,Q)$.

Indeed, let $(a_0,b_0,a_1,b_1,\dots,a_n,b_n,\dots)$ be a play of
$\mathsf{G}_\gamma(P,Q)$ in which ONE follows the strategy
$\varphi$, and let $X\in[\omega]^{\aleph_0}$ be arbitrary. Write
$X=\{x_k:k\in\omega\}$ with $x_k<x_{k+1}$ for all $k\in\omega$.
As $\trianglelefteq$ is downwards $P$-compatible, it follows that
$$
(
\sigma(\emptyset),b_{x_0},\sigma((b_{x_0})),b_{x_1},\sigma((b_{x_0},b_{x_1})),b_{x_2},\dots,
\sigma((b_{x_i})_{i<k}),b_{x_k},\dots
)
$$
is a play of $\mathsf{G}(P,Q)$ in which ONE makes use of the
winning strategy $\sigma$, since for each $k\in\omega$ we have
$$
\sigma((b_{x_i})_{i<k})\trianglelefteq
\tilde{a}({\{\sigma((b_{s(i)})_{i\in\mathrm{dom}(s)}):s\in S_{x_k}\}})
=\varphi((b_j)_{j<x_k})=a_{x_k}
$$
and
$
a_{x_k}Rb_{x_k}.
$
Thus $\{b_{x_k}:k\in\omega\}\in\mathrm{Dom}(Q)$; since $X$ was chosen
arbitrarily, it follows that
$(b_n)_{n\in\omega}\in\mathrm{Dom}_\gamma(Q)$.
\end{proof}

\begin{corol}[Gruenhage \cite{gruen76}]
\label{equiv.W}

Let $X$ be a topological space and $x\in X$. The following statements
are equivalent:

\begin{itemize}

\item[$(a)$]
ONE has a winning strategy in the game $\mathsf{G}^c_{O,P}(X,x)$ (see
Example \ref{game.gruen});

\item[$(b)$]
ONE has a winning strategy in the game $\mathsf{G}_{O,P}(X,x)$ (see
paragraph preceding Corollary \ref{dual.o-p_x}).

\end{itemize}

\end{corol}

\begin{proof}

Apply Theorem \ref{gn} with $P=Q=(\tau_x,X,\ni)$ and $V\trianglelefteq
W\leftrightarrow V\supseteq W$.
\end{proof}

Before stating the next result, we evoke Corollaries 4.3 and 4.4 of
\cite{telg1}, which can be combined in a single statement as:

\begin{prop}[Telg\'arsky \cite{telg1}]
\label{finite-open}

The point-open game is equivalent to the finite-open game, which is
played according to the following rules: In each inning $n\in\omega$,
ONE picks a finite subset $F_n$ of $X$, and then TWO chooses
$U_n\in\tau$ with $F_n\subseteq U_n$; ONE wins if
$\{U_n:n\in\omega\}\in\mathcal{O}_X$, and TWO wins otherwise.

\end{prop}

With Proposition \ref{finite-open} in mind, we recall the
\emph{strict point-open game}, introduced in \cite{g-n}. The game is
played according to the same rules as the finite-open game,\footnote{In
view of Proposition \ref{finite-open}, the authors make no distinction
between the point-open game and the finite-open game in \cite{g-n}.}
but now ONE wins if
$X=\bigcup_{n\in\omega}\bigcap_{m\in\omega\setminus n}U_m$.

\begin{corol}[Gerlits--Nagy \cite{g-n}]
\label{equiv-g-n}

Let $X$ be a topological space. The following statements
are equivalent:

\begin{itemize}

\item[$(a)$]
ONE has a winning strategy in the point-open game on $X$;

\item[$(b)$]
ONE has a winning strategy in the strict point-open game on $X$.

\end{itemize}

\end{corol}

\begin{proof}

Apply Theorem \ref{gn} with $P=([X]^{<\aleph_0},\tau,\subseteq)$,
$Q=(X,\tau,\in)$ and $F\trianglelefteq G\leftrightarrow F\subseteq G$.
\end{proof}

For the next two corollaries, we recall the game $\mathsf{G}^*(X)$
introduced by Gruenhage in \cite{eberlein} for every noncompact space
$X$. In each inning $n\in\omega$ of $\mathsf{G}^*(X)$, ONE picks a
compact set $C_n\subseteq X$, and then TWO picks a nonempty compact
set $L_n\subseteq X$ with $C_n\cap L_n=\emptyset$. ONE wins if the
family $\{L_n:n\in\omega\}$ is locally finite, and loses otherwise.
We shall also consider a variation $\mathsf{G}^{**}(X)$ of this game,
which is played according to the same rules but now ONE wins if and
only if $\{L_n:n\in\omega\}\in\mathfrak{M}_X$.

\begin{corol}
\label{moving-off-game}

Let $X$ be a noncompact locally compact space.
The following assertions are equivalent:

\begin{itemize}

\item[$(a)$]
ONE has a winning strategy in $\mathsf{G}^{*}(X)$;

\item[$(b)$]
ONE has a winning strategy in $\mathsf{G}^{**}(X)$.

\end{itemize}

\end{corol}

\begin{proof}

For $(b)\rightarrow(a)$, let $\sigma$ be a winning strategy for ONE in
$\mathsf{G}^{**}(X)$. We can define a strategy $\varphi$ for
ONE in $\mathsf{G}^*(X)$ by setting
$\varphi((L_i)_{i<n})=\sigma((L_i)_{i<n})\cup\bigcup_{i<n}L_i$ for all
$(L_i)_{i<n}\in\mbox{}^{<\omega}(K(X)\setminus\{\emptyset\})$. Note
that $\varphi$ is a winning strategy since the set of TWO's moves in a
play in which ONE follows $\varphi$ is an infinite locally finite
family of nonempty compact sets, and hence is a moving-off family by
Lemma 4 of \cite{paracompactness}.

For $(a)\rightarrow(b)$, apply Theorem \ref{gn} with
$P=Q=(K(X),K(X)\setminus\{\emptyset\},R)$, where $CRL\leftrightarrow
C\cap L=\emptyset$. The result follows from the observation that, if
$X$ is locally compact and
$(L_n)_{n\in\omega}\in\mathrm{Dom}_\gamma(Q)$, then
$\{L_n:n\in\omega\}$ is locally finite.
\end{proof}

The next corollary presents some variations on the game-theoretic
characterization of paracompactness for locally compact $T_2$ spaces
obtained by Gruenhage in \cite{eberlein} --- which states that a
locally compact $T_2$ space is paracompact if and only if ONE has a
winning strategy in the game $\mathsf{G}^*(X)$.

\begin{corol}
\label{gruen-mop}

Let $X$ be a noncompact locally compact $T_2$ space and
$\mathfrak{L}=\{\mathcal{L}\subseteq
K(X)\setminus\{\emptyset\}:\mathcal{L}$ is locally finite$\}$. The
following conditions are equivalent:

\begin{itemize}

\item[$(a)$]
$X$ is paracompact;

\item[$(b)$]
ONE has a winning strategy in $\mathsf{G}^*(X)$;

\item[$(c)$]
ONE has a winning strategy in $\mathsf{G}^{**}(X)$;

\item[$(d)$]
TWO has a winning strategy in
$\mathsf{G}_1(\mathfrak{M},\mathfrak{L})$;

\item[$(e)$]
TWO has a winning strategy in
$\mathsf{G}_1(\mathfrak{M},\mathfrak{M})$.

\end{itemize}

\end{corol}

\begin{proof}

We have just proven $(b)\leftrightarrow(c)$ in Corollary
\ref{moving-off-game}. The equivalences
$(a)\leftrightarrow(b)$ and $(b)\leftrightarrow(d)$ are
Theorem 5 of \cite{eberlein} and Theorem 2 of \cite{paracompactness}
respectively. Finally, $(c)\leftrightarrow(e)$ follows from Theorem
\ref{dual} with $P=Q=(K(X),K(X)\setminus\{\emptyset\},R)$, where
$CRL\leftrightarrow C\cap L=\emptyset$.
\end{proof}

\section{$\aleph_0$-modifications}

In this section, we study another variation of the game
$\mathsf{G}(P,Q)$ (resp. $\mathsf{G}_1(\mathrm{Dom}(P),\mathrm{Dom}(Q))$)
for which the existence of a winning strategy for player ONE (resp. TWO)
although apparently stronger, turns out to be equivalent to
the same condition for the original game.

This variation will be defined in terms of the following concept.

\begin{defin}
\label{def.aleph_0-modif}

The \emph{$\aleph_0$-modification} of a relation $P=(A,B,R)$ is the
relation
$P_{\aleph_0}=(A,[B]^{\le\aleph_0}\setminus\{\emptyset\},\widetilde
R)$, where $a\widetilde R E\leftrightarrow\forall b\in E\;(aRb)$.

\end{defin}

The equivalence previously mentioned can then be stated as follows.

\begin{prop}
\label{equiv.game-aleph_0}

Let $P$ and $Q$ be relations. The following conditions are equivalent:
\begin{itemize}
\item[$(a)$]
ONE has a winning strategy in the game $\mathsf{G}(P,Q)$;
\item[$(b)$]
ONE has a winning strategy in the game
$\mathsf{G}(P_{\aleph_0},Q_{\aleph_0})$.
\end{itemize}

\end{prop}

\begin{proof}

Write $P=(A,B,R)$ and $Q=(C,D,T)$.

The implication $(b)\rightarrow(a)$ is immediate, since
$\mathsf{G}(P,Q)$ is equivalent to the game
$\mathsf{G}(P_{\aleph_0},Q_{\aleph_0})$ played with the additional
restriction that TWO must choose one-element subsets of $B$.

For $(a)\rightarrow(b)$, let $\sigma:\mbox{}^{<\omega}B\rightarrow A$
be a winning strategy for ONE in $\mathsf{G}(P,Q)$.
Fix an injective function $s\mapsto m_s$ from
$\mbox{}^{<\omega}\omega$ to $\omega$ satisfying $s\subseteq
t\rightarrow m_s\le m_t$ for every $s,t\in\mbox{}^{<\omega}\omega$ ---
for example, define $m_s=\prod_{i\in\mathrm{dom}(s)}p_i^{s(i)+1}$,
where $p_i$ is the $i$-th prime number.
Now write each $E\in[B]^{\le\aleph_0}\setminus\{\emptyset\}$ as
$E=\{b^E_k:k\in\omega\}$, and let $\tilde{a}\in A$ be fixed.
Define
$\varphi:\mbox{}^{<\omega}([B]^{\le\aleph_0}\setminus\{\emptyset\})\rightarrow
A$
by
$$
\varphi((E_j)_{j<n})=\left\{\begin{array}{lll}
\sigma\left(\left(
b^{E_{m_{s\upharpoonright i}}}_{s(i)}
\right)_{i\in\mathrm{dom}(s)}
\right),&&\textrm{ if }n=m_s;\\
&&\\
\tilde{a},&&\textrm{ otherwise.}
\end{array}\right.
$$
We claim that $\varphi$ is a winning strategy for ONE in
$\mathsf{G}(P_{\aleph_0},Q_{\aleph_0})$.

Indeed, let $(a_0,E_0,a_1,E_1,\dots)$ be a play of
$\mathsf{G}(P_{\aleph_0},Q_{\aleph_0})$ in which ONE follows
the strategy $\varphi$, and suppose that there is $c\in C$ such that
$c\widetilde TE_n$ does not hold for any $n\in\omega$. Define
$f:\omega\rightarrow\omega$ by recursively choosing $f(i)\in\omega$
such that $cTb^{E_{m_{f\upharpoonright i}}}_{f(i)}$ does not
hold. Then we obtain a contradiction from the fact that
\begin{center}
{\renewcommand{\arraystretch}{3.0}
\renewcommand{\tabcolsep}{0.25cm}
\begin{tabular}{c|c}
ONE & TWO \\

\hline

$\sigma(\emptyset)$ & $b^{E_{m_\emptyset}}_{f(0)}$\\
$\sigma\left(\left(b^{E_{m_\emptyset}}_{f(0)}\right)\right)$ &
$b^{E_{m_{(f(0))}}}_{f(1)}$\\
$\sigma\left(\left(b^{E_{m_\emptyset}}_{f(0)},b^{E_{m_{(f(0))}}}_{f(1)}\right)\right)$
& $b^{E_{m_{(f(0),f(1))}}}_{f(2)}$\\
$\vdots$ & $\vdots$ \\
$\sigma\left(\left(b^{E_{m_{f\upharpoonright j}}}_{f(j)}\right)_{j<i}\right)$ &
$b^{E_{m_{f\upharpoonright i}}}_{f(i)}$\\
$\vdots$ & $\vdots$
\end{tabular}}\\
\end{center}
is a play of $\mathsf{G}(P,Q)$ in which ONE makes use of the
winning strategy $\sigma$ and loses --- since none of TWO's
moves dominate $c$ in $T$.
\end{proof}

In view of Theorem \ref{dual}, we can rewrite Proposition
\ref{equiv.game-aleph_0} as:

\begin{corol}
\label{G1-aleph_0}

Let $P$ and $Q$ be relations. The following conditions are equivalent:
\begin{itemize}
\item[$(a)$]
TWO has a winning strategy in the game
$\mathsf{G}_1(\mathrm{Dom}(P),\mathrm{Dom}(Q))$;
\item[$(b)$]
TWO has a winning strategy in the game
$\mathsf{G}_1(\mathrm{Dom}(P_{\aleph_0}),\mathrm{Dom}(Q_{\aleph_0}))$.
\end{itemize}

\end{corol}

The following result is Theorem 5.1 of \cite{telg2}. Given a nonempty
family $\mathbf{K}$ of subsets of a topological space $X$, we call
\emph{$\mathbf{K}$-open game on $X$} the game in which, in each inning
$n\in\omega$, ONE chooses $K_n\in\mathbf{K}$ and then TWO
picks an open set $U_n\subseteq X$ with $K_n\subseteq U_n$; the
winner is ONE if $X=\bigcup_{n\in\omega}U_n$, and TWO
otherwise. The \emph{$\mathbf{K}$-$G_\delta$ game on $X$} is played
according to the same rules, replacing ``open'' with ``$G_\delta$''.

\begin{corol}[Telg\'arsky \cite{telg2}]
\label{telg83}

Let $X$ be a topological space and $\mathbf{K}$ be a nonempty family
of subsets of $X$. The following conditions are equivalent:
\begin{itemize}
\item[$(a)$]
ONE has a winning strategy in the $\mathbf{K}$-open game on $X$;
\item[$(b)$]
ONE has a winning strategy in the $\mathbf{K}$-$G_\delta$ game
on $X$.
\end{itemize}

\end{corol}

\begin{proof}

Apply Proposition \ref{equiv.game-aleph_0} with
$P=(\mathbf{K},\tau,\subseteq)$ and $Q=(X,\tau,\in)$.
The result follows from the observation that the
games $\mathbf{K}$-$G_\delta$ and
$\mathsf{G}(P_{\aleph_0},Q_{\aleph_0})$ are equivalent.
\end{proof}

As another consequence of Proposition \ref{equiv.game-aleph_0},
we have:

\begin{corol}
\label{supertight-game}

Let $X$ be a topological space and $x\in X$. The following conditions
are equivalent:
\begin{itemize}
\item[$(a)$]
TWO has a winning strategy in the game
$\mathsf{G}_1(\Omega_x,\Omega_x)$;
\item[$(b)$]
TWO has a winning strategy in the game
$\mathsf{G}_1(\pi\mathcal{N}_{x},\pi\mathcal{N}_{x})$;
\item[$(c)$]
TWO has a winning strategy in the game
$\mathsf{G}_1(\pi\mathcal{N}^{\aleph_0}_{x},\pi\mathcal{N}^{\aleph_0}_{x})$.
\end{itemize}

\end{corol}

\begin{proof}

It is clear that $(c)\rightarrow(b)\rightarrow(a)$. Now the
equivalence between $(a)$ and $(c)$ follows from Corollary
\ref{G1-aleph_0} with $P=Q=(\tau_x,X,\ni)$.
\end{proof}

As an immediate consequence of Corollary
\ref{supertight-game} (see also Corollary \ref{supertight-equiv}), we
have the following result, which answers Question 4.9 of \cite{ab} in
the affirmative.

\begin{corol}
\label{answer.aurichi-bella}

Let $X$ be a topological space and $x\in X$.
If TWO has a winning strategy in the game
$\mathsf{G}_1(\Omega_x,\Omega_x)$ on $X$, then
$\pi\mathcal{N}^{\aleph_0}_x$-Lindel\"of holds.

\end{corol}


\section{$\aleph_0$-preserving relations}
\label{sec.aleph_0-preserving}

Inspired by some features of relations of the form $P_{\aleph_0}$
introduced in the previous section (see Lemma
\ref{P_aleph_0-preserving}), we will now aim at finding
general conditions on relations under which a Lindel\"of-like
property turns out to be strong enough to yield the nonexistence
of a winning strategy for ONE in the associated selective game.

\begin{defin}
\label{countably.compatible}

Let $P=(A,B,R)$ be a relation and $\preceq$ be a partial order on
$B$. We say that $\preceq$ is
\emph{countably downwards $P$-compatible} if, for every $a\in A$ and
every $E\in[B]^{\le\aleph_0}\setminus\{\emptyset\}$,
$$
(\forall b\in E\;(aRb))\rightarrow\exists\tilde{b}\in
B\;(aR\tilde{b}\;\&\;\forall b\in E\;(\tilde{b}\preceq b)).
$$

\end{defin}

\begin{defin}
\label{aleph_0-preserving}

A relation $P=(A,B,R)$ is \emph{$\aleph_0$-preserving} if there is a
partial order $\preceq$ on $B$ that is both upwards $P$-compatible
(see Definition \ref{def.order}) and countably downwards
$P$-compatible.

\end{defin}

Our main examples of $\aleph_0$-preserving relations will be of
the form $P_{\aleph_0}$:

\begin{lemma}
\label{P_aleph_0-preserving}

Let $P$ be a relation. Then $P_{\aleph_0}$ (see Definition
\ref{def.aleph_0-modif}) is an $\aleph_0$-preserving relation.

\end{lemma}

\begin{proof}

Just note that, if $P=(A,B,R)$, then the partial order $\preceq$ on
$[B]^{\le\aleph_0}$ defined by $E_1\preceq E_2\leftrightarrow
E_1\supseteq E_2$ witnesses the fact that $P_{\aleph_0}$ is
$\aleph_0$-preserving.
\end{proof}

The following proposition is the main result of this section.

\begin{prop}
\label{equiv.aleph_0-preserving}

Let $P$ be an $\aleph_0$-preserving relation. The following conditions
are equivalent:
\begin{itemize}
\item[$(a)$]
$\mathrm{Dom}(P)$-Lindel\"of;
\item[$(b)$]
$\mathsf{S}_1(\mathrm{Dom}(P),\mathrm{Dom}(P))$;
\item[$(c)$]
ONE does not have a winning strategy in the game
$\mathsf{G}_1(\mathrm{Dom}(P),\mathrm{Dom}(P))$.
\end{itemize}
\end{prop}

\begin{proof}

Clearly, $(c)\rightarrow(b)\rightarrow(a)$. We will prove the
implication $(a)\rightarrow(c)$.

Let a strategy for ONE in
$\mathsf{G}_1(\mathrm{Dom}(P),\mathrm{Dom}(P))$ be fixed. By $(a)$, we
may assume that each of ONE's moves in this strategy is a
countable set; this allows us to regard such strategy as an indexed
family $(b_t)_{t\in\mbox{}^{<\omega}\omega\setminus\{\emptyset\}}$ ---
meaning that, if $s\in\mbox{}^n\omega$ is such that
TWO's choices in the first $n$ innings were
$(b_{s\upharpoonright i})_{i<n}$, then
ONE's move in the $n$-th inning is
$\{b_{s^\smallfrown(k)}:k\in\omega\}\in\mathrm{Dom}(P)$.

Now write $P=(A,B,R)$, and let $\preceq$ be a partial order on $B$
witnessing the fact that $P$ is $\aleph_0$-preserving.
For each $a\in A$ and each function
$F:\mbox{}^{<\omega}\omega\rightarrow\omega$, let then
$\tilde{b}^F_a\in B$ be such that $aR\tilde{b}^F_a$ and $\forall
s\in\mbox{}^{<\omega}\omega\;(\tilde{b}^F_a\preceq
b_{s^\smallfrown(F(s))})$.
Note that
$\{\tilde{b}^F_a:a\in
A,F\in\mbox{}^{(\mbox{}^{<\omega}\omega)}\omega\}\in\mathrm{Dom}(P)$;
by $(a)$, it follows that there exist $\{a_n:n\in\omega\}\subseteq A$
and
$\{F_n:n\in\omega\}\subseteq\mbox{}^{(\mbox{}^{<\omega}\omega)}\omega$
with $\{\tilde{b}^{F_n}_{a_n}:n\in\omega\}\in\mathrm{Dom}(P)$. Now
define $f:\omega\rightarrow\omega$ recursively by
$f(n)=F_n(f\upharpoonright n)$ for each $n\in\omega$. We claim that
$\{b_{f\upharpoonright(n+1)}:n\in\omega\}\in\mathrm{Dom}(P)$ --- which
shows that the strategy at hand for ONE can be defeated.

In order to see this, let $a\in A$ be arbitrary. Since
$\{\tilde{b}^{F_n}_{a_n}:n\in\omega\}\in\mathrm{Dom}(P)$, there is
$m\in\omega$ such that $aR\tilde{b}^{F_m}_{a_m}$; thus, as
$\tilde{b}^{F_m}_{a_m}\preceq b_{s^\smallfrown(F(s))}$ holds for
$s=f\upharpoonright m$ in particular, it follows from the equality
$f(m)=F_m(f\upharpoonright m)$ and the hypothesis that $\preceq$ is
upwards $P$-compatible that $aRb_{f\upharpoonright(m+1)}$, as required.
\end{proof}

A topological space is \emph{strongly Alster} \cite{renan} if
$\mathcal{G}_K$-Lindel\"of holds, where
$\mathcal{G}_K=\{\mathcal{W}:(\forall W\in\mathcal{W}\;(W$ is a
$G_\delta$ subset of $X))$ $\&\;(\forall C\in K(X)\;\exists
W\in\mathcal{W}\;(C\subseteq W))\}$.

\begin{corol}
\label{strongly.alster}

The following statements are equivalent for a topological space $X$:

\begin{itemize}

\item[$(a)$]
$X$ is strongly Alster;

\item[$(b)$]
$\mathsf{S}_1(\mathcal{G}_K,\mathcal{G}_K)$;

\item[$(c)$]
ONE does not have a winning strategy in the game
$\mathsf{G}_1(\mathcal{G}_K,\mathcal{G}_K)$.

\end{itemize}

\end{corol}

\begin{proof}

Apply Proposition \ref{equiv.aleph_0-preserving} with
$P=(K(X),G_\delta(X),\subseteq)$, where $K(X)=\{C\subseteq X:C$ is
compact$\}$ and $G_\delta(X)=\{W\subseteq X:W$ is a countable
intersection of open sets$\}$. (Note that $W_1\preceq
W_2\leftrightarrow W_1\subseteq W_2$ witnesses that $P$ is
$\aleph_0$-preserving.)
\end{proof}

The next corollary deals with a game that was also explored in
Corollary \ref{supertight-game}.

\begin{corol}
\label{supertight-equiv}

Let $X$ be a topological space and $x\in X$. The following conditions
are equivalent:

\begin{itemize}

\item[$(a)$]
$\pi\mathcal{N}^{\aleph_0}_x$-Lindel\"of holds;

\item[$(b)$]
$\mathsf{S}_1(\pi\mathcal{N}^{\aleph_0}_x,\pi\mathcal{N}^{\aleph_0}_x)$;

\item[$(c)$]
ONE does not have a winning strategy in the game
$\mathsf{G}_1(\pi\mathcal{N}^{\aleph_0}_x,\pi\mathcal{N}^{\aleph_0}_x)$.

\end{itemize}

\end{corol}

\begin{proof}

Note that $P=(\tau_x,[X]^{\le\aleph_0},\supseteq)$
is $\aleph_0$-preserving by Lemma \ref{P_aleph_0-preserving}. Now
apply Proposition \ref{equiv.aleph_0-preserving}.
\end{proof}

We note that the equivalence between $(a)$ and $(b)$ in
Corollary \ref{supertight-equiv} also follows from Proposition 2.5(2)
of \cite{bella-sakai}.


\section{Remarks}

\paragraph{A}

As not all topological properties can be expressed in terms of
relations, it should be made clear that there are selective
topological games that have been studied in the literature for which
the analogue of the previous results does not hold. We illustrate this
with the following selective property: A topological space
$X$ is \emph{selectively screenable} \cite{addis} if, for every
sequence $(\mathcal{U}_n)_{n\in\omega}$ of open covers of $X$, there
is a sequence $(\mathcal{V}_n)_{n\in\omega}$ of families of open
subsets of $X$ such that $X=\bigcup_{n\in\omega}\mathcal{V}_n$ and
each $\mathcal{V}_n$ is a pairwise disjoint partial refinement of
$\mathcal{U}_n$. It follows from Example 1 of
\cite{e.pol} and Theorem 2.2 of \cite{bab} that TWO having a winning
strategy in the game naturally associated with selective screenability
does not imply that the space is productively selectively screenable;
therefore, a result similar to Propositions \ref{G1} and \ref{Gfin}
could not be obtained for this concept.\\

\paragraph{B}

Having in mind the properties $\mathsf{E}_1$--$\mathsf{E}_4$ from the
Introduction, there seems to be a gap in Propositions \ref{G1}
and \ref{Gfin}, which motivates the two main open questions of this
paper:

\begin{prob}
\label{G1?}

Let $P$, $P'$, $Q$ and $Q'$ be relations such that
that TWO has a winning strategy in the game
$\mathsf{G}_1(\mathrm{Dom}(P'),\mathrm{Dom}(Q'))$ and
ONE does not have a winning strategy in the game
$\mathsf{G}_1(\mathrm{Dom}(P),\mathrm{Dom}(Q))$.
Does it follow that ONE does not have a winning strategy in the game 
$\mathsf{G}_1(\mathrm{Dom}(P\otimes P'),\mathrm{Dom}(Q\otimes Q'))$?

\end{prob}

\begin{prob}
\label{Gfin?}

Let $P=(A,B,R)$, $P'=(A',B',R')$, $Q=(C,D,T)$
and $Q'=(C',D',T')$ be relations with $B=D$
such that:
\begin{itemize}
\item[$\cdot$]
there is a partial order $\preceq$ on $B$ that is both
downwards $P$-compatible and upwards $Q$-compatible;
\item[$\cdot$]
TWO has a winning strategy in the game
$\mathsf{G}_{\mathrm{fin}}(\mathrm{Dom}(P'),\mathrm{Dom}(Q'))$; and
\item[$\cdot$]
ONE does not have a winning strategy in the game
$\mathsf{G}_{\mathrm{fin}}(\mathrm{Dom}(P),\mathrm{Dom}(Q))$.
\end{itemize}
Does it follow that ONE does not have a winning strategy in the game
$\mathsf{G}_{\mathrm{fin}}(\mathrm{Dom}(P\otimes
P'),\mathrm{Dom}(Q\otimes Q'))$?

\end{prob}

It should be pointed out that,
in many instances of the topological properties in which we are
interested in this paper, it is the case that ONE does not have a
winning strategy in the game
$\mathsf{G}_1(\mathrm{Dom}(P),\mathrm{Dom}(Q))$ if and
only if $\mathsf{S}_1(\mathrm{Dom}(P),\mathrm{Dom}(Q))$ holds
(and similarly for $\mathsf{G}_{\mathrm{fin}}$ and
$\mathsf{S}_{\mathrm{fin}}$); see e.g. \cite[Theorem 10]{hur} (for
the Menger game), \cite[Lemma 2]{pawl} (for the Rothberger game) and
\cite[Theorems 2 and 14]{sch5} (for the games
$\mathsf{G}_{\mathrm{fin}}(\mathcal{D},\mathcal{D})$ and
$\mathsf{G}_1(\mathcal{D},\mathcal{D})$).
As a consequence, none of these
instances could provide us with a negative answer to Problems
\ref{G1?} and \ref{Gfin?}.

It is also known that there are other instances
in which this equivalence does not hold
--- such as $\mathsf{G}_1(\mathfrak{D},\mathfrak{D})$
\cite[Example 3]{sch6} ---,
which could be a first attempt to answer Problem \ref{G1?} in the
negative. More explicitly:

\begin{prob}
\label{G1(D,D)?}

Let $X$ and $Y$ be topological spaces such that
TWO has a winning strategy in the game
$\mathsf{G}_1(\mathfrak{D}_X,\mathfrak{D}_X)$ and
ONE does not have a winning strategy in the game
$\mathsf{G}_1(\mathfrak{D}_Y,\mathfrak{D}_Y)$.
Does it follow that ONE does not have a winning strategy in the game 
$\mathsf{G}_1(\mathfrak{D}_{X\times Y},\mathfrak{D}_{X\times Y})$?

\end{prob}

\paragraph{C}

Regarding the four properties $\mathsf{E}_i$ mentioned in the
Introduction, and having in mind Proposition \ref{G1}, one could
ask whether, for some pair $(i,j)$ with $2\le i\le j\le 4$, it is the
case that every relation in the class $\mathsf{E}_i$ is productively
$\mathsf{E}_j$.
This possibility can be ruled out by considering the following.

For a topological space $X$ and $P=Q=(X,\tau,\in)$, we have:
\begin{itemize}
\item[$\mathsf{E}_2$:]
ONE does not have a winning strategy in the game
$\mathsf{G}_1(\mathcal{O}_X,\mathcal{O}_X)$;
\item[$\mathsf{E}_3$:]
$\mathsf{S}_1(\mathcal{O}_X,\mathcal{O}_X)$;
\item[$\mathsf{E}_4$:]
$X$ is a Lindel\"of space.
\end{itemize}

By Lemma 2 of \cite{pawl}, $\mathsf{E}_2$ and $\mathsf{E}_3$ are
equivalent. In Theorem 8 of \cite{stevo}, two Rothberger spaces are
constructed in such a way that their product is not Lindel\"of. The
conjunction of these results shows that $\mathsf{E}_2$ is not strong
enough to imply productivity with respect to either $\mathsf{E}_2$,
$\mathsf{E}_3$ or $\mathsf{E}_4$.

\section*{Acknowledgements}

This research was done during a visit of the first author to the
Department of Mathematics at Boise State University. The author wishes
to express his gratitude to the Department for their hospitality
and academic support.

We are deeply indebted to Samuel Coskey, whose insightful comments led
us to the approach of dominating families here presented.
We would also like to thank Leandro Aurichi, Liljana Babinkostova
and Bruno Pansera for many helpful discussions that greatly influenced
this work.

%

\end{document}